\documentclass{amsart}
\usepackage{amssymb,amsmath}
\usepackage{graphicx}
\usepackage{xcolor}

\theoremstyle{plain}
\newtheorem{thm}{Theorem}[section]

\newtheorem{cor}[thm]{Corollary}
\newtheorem{lem}[thm]{Lemma}
\newtheorem{prop}[thm]{Proposition}

\theoremstyle{definition}
\newtheorem{defn}{Definition}[section]

\theoremstyle{remark}
\newtheorem{rem}{Remark}[section]
\newtheorem{ex}{Example}[section]

\numberwithin{equation}{section}

\newcommand{\ra}{\rightarrow}
\newcommand{\Ra}{\Rightarrow}

\newcommand{\Lra}{\Leftrightarrow}

\begin{document}

\title{Cutoffs for product chains}

\author[G.-Y. Chen]{Guan-Yu Chen$^1$}

\author[T. Kumagai]{Takashi Kumagai$^2$}

\address{$^1$Dept. of Appl. Math., National Chiao Tung University, Hsinchu 300, Taiwan}
\email{gychen@math.nctu.edu.tw}

\address{$^2$RIMS, Kyoto University, Kyoto 606-8502, Japan}\email{kumagai@kurims.kyoto-u.ac.jp}



\begin{abstract}
In this article, we consider products of ergodic Markov chains and discuss their cutoffs in the total variation. Through
an
inequality relating the total variation and the Hellinger distance, we may identify the total variation cutoffs with cutoffs in the Hellinger distance. This provides a new scheme to study the total variation mixing of Markov chains, in particular, product chains. In the theoretical framework, a series of criteria are introduced to examine cutoffs and a comparison of mixing between the product chain and its coordinate chains is made in detail. For illustration, we consider products of two-state chains, cycles and other typical examples.
\end{abstract}

\maketitle

\section{Introduction}

Let $\mathcal{X}$ be a countable set, $K$ be an irreducible stochastic matrix indexed by $\mathcal{X}$ and $\pi$ be a probability on $\mathcal{X}$. We write the triple $(\mathcal{X},K,\pi)$ for a discrete time Markov chain on $\mathcal{X}$ with transition matrix $K$ and stationary distribution $\pi$. It is well-known that if $K$ is aperiodic, then $K^m(x,y)$ converges to $\pi(y)$ as $m$ tends to infinity for all $x,y\in\mathcal{X}$. To quantize the convergence of $K^m$ to $\pi$, we consider the (maximum) total variation and the (maximum) Hellinger distance, which are defined by
\begin{equation}\label{d-tv}
 d_{\text{\tiny TV}}(m):=\sup_{x\in\mathcal{X},A\subset\mathcal{X}}\{K^m(x,A)-\pi(A)\},
\end{equation}
and
\begin{equation}\label{d-hd}
 d_H(m):=\sup_{x\in\mathcal{X}}\left(\frac{1}{2}\sum_{y\in\mathcal{X}}
 \left(\sqrt{K^m(x,y)}-\sqrt{\pi(y)}\right)^2\right)^{1/2}.
\end{equation}
As the above distances are non-increasing in $m$, it is natural to consider the mixing times of $d_{\text{\tiny TV}}$ and $d_H$, which are respectively defined by
\[
 T_{\text{\tiny TV}}(\epsilon):=\inf\{m\ge 0|d_{\text{\tiny TV}}(m)\le\epsilon\},
 \quad T_H(\epsilon):=\inf\{m\ge 0|d_H(m)\le\epsilon\}.
\]

For the weak convergence of distributions, the total variation arose naturally from the view point of probability, while the importance of the Hellinger distance is exemplified from the proof of Kakutani's dichotomy theorem in \cite{K48} for the study of infinite product measures.
The following inequalities provide a comparison of the total variation and the Hellinger distance, which
are corollaries in \cite{Sh96} (see (25) on p.365 for the details) and say
\begin{equation}\label{eq-hdtv}
 1-\sqrt{1-d_{\text{\tiny TV}}^2(m)}\le d_H^2(m)\le d_{\text{\tiny TV}}(m).
\end{equation}
As a consequence, one obtains from (\ref{eq-hdtv}) the following comparison of mixing times,
\begin{equation}\label{eq-mixhdtv}
 T_{\text{\tiny TV}}(\epsilon\sqrt{2-\epsilon^2})\le T_H(\epsilon)\le T_{\text{\tiny TV}}(\epsilon^2),\quad\forall\epsilon\in(0,1).
\end{equation}
We can further compare the cutoffs, introduced below, in the total variation and the Hellinger distance. Such a comparison will play a key role through this article.

In this article, we focus on the study of product chains and their cutoffs. To see a definition of product chains, let $(\mathcal{X}_i,K_i,\pi_i)_{i=1}^n$ be irreducible Markov chains and set
\begin{equation}\label{eq-prodxpi}
 \mathcal{X}=\mathcal{X}_1\times\cdots\times\mathcal{X}_n,\quad \pi=\pi_1\times\cdots\times \pi_n,
\end{equation}
and
\begin{equation}\label{eq-prodk}
 K=\sum_{i=1}^np_iI_1\otimes\cdots\otimes I_{i-1}\otimes K_i\otimes I_{i+1}\otimes\cdots\otimes I_n,
\end{equation}
where $I_j$ is the identity matrix indexed by $\mathcal{X}_j$, $A\otimes B$ denotes the tensor product of matrices $A,B$ and $p_1,...,p_n$ are positive reals satisfying $p_1+\cdots+p_n=1$. It is obvious that $K$ is a transition matrix on $\mathcal{X}$ with stationary distribution $\pi$. Thereafter, we call $(\mathcal{X},K,\pi)$ the product chain of $(\mathcal{X}_i,K_i,\pi_i)_{i=1}^n$ according to the probability vector $(p_1,...,p_n)$. As the product chain $K^m$ has no simple expression, say in a formula of $(K^m_i)_{i=1}^n$, the study of its total variation and Hellinger distance can be challenging. However, when the diagonal entries in a transition matrix are bounded below by a positive constant, its mixing time is comparable with the mixing time of its associated continuous time Markov chain. As discussed below, when we consider product chains, it is more convenient to use continuous time Markov chains rather than discrete time ones. For a comparison of discrete and continuous time chains, see e.g. \cite{CSal13-1} for an early reference and Proposition \ref{p-comp2} for another.

For a discrete time chain $(\mathcal{X},K,\pi)$, let the triple $(\mathcal{X},H_t,\pi)$ be such that
$H_t=e^{-t(I-K)}$. Note that, if $(X_m)_{m=0}^\infty$ is a realization of $(\mathcal{X},K,\pi)$ and $(N_t)_{t\ge 0}$ is a Poisson process (with parameter $1$) independent of $(X_m)_{m=0}^\infty$, then $(X_{N_t})_{t\ge 0}$ is a continuous time Markov chain on $\mathcal{X}$ with transition matrices $(H_t)_{t\ge 0}$. Here, we write $(\mathcal{X},H_t,\pi)$ for $(X_{N_t})_{t\ge 0}$ and call it the continuous time Markov chain associated with $(\mathcal{X},K,\pi)$. To study the convergence of $(\mathcal{X},H_t,\pi)$, one may replace $K^m$ with $H_t$ in (\ref{d-tv}) and (\ref{d-hd}) to achieve its total variation and Hellinger distance, while the associated mixing times are defined in a similar way. By Lemma \ref{l-comp}, (\ref{eq-hdtv}) and (\ref{eq-mixhdtv}) are also valid in the continuous time case.
We write $d,T$ for the distance and mixing time of $(\mathcal{X},K,\pi)$, and write $d^{(c)},T^{(c)}$ for those of $(\mathcal{X},H_t,\pi)$.

Back to the product chain in (\ref{eq-prodxpi})-(\ref{eq-prodk}), let $(\mathcal{X}_i,H_{i,t},\pi_i)$ and $(\mathcal{X},H_t,\pi)$ be the continuous time chains associated with $(\mathcal{X}_i,K_i,\pi)$ and $(\mathcal{X},K,\pi)$. It follows immediately from the previous setting that
\begin{equation}\label{eq-prodcts}
 H_t=H_{1,p_1t}\otimes\cdots\otimes H_{n,p_nt}.
\end{equation}
In general, there is no similar form for $K^m$, and that is the reason we use continuous time Markov chains. Through (\ref{eq-prodcts}), one may express the Hellinger distance of $(\mathcal{X},H_t,\pi)$ as a formula of the Hellinger distance of $(\mathcal{X}_i,H_{i,t},\pi_i)$. See \cite[Exercise 20.5]{LPW08} for one version and also (\ref{eq-prodhd}) in Lemma \ref{l-prodmixing} for another. Note that the equality in (\ref{eq-prodhd}) can fail in the total variation but, along with (\ref{eq-hdtv}) and (\ref{eq-mixhdtv}), the total variation of $(\mathcal{X},H_t,\pi)$ can be closely related to the total variation of $(\mathcal{X}_i,H_{i,t},\pi_i)$ and this is discussed
in detail in Section \ref{s-dpc}.

The cutoff phenomenon of Markov chains was introduced by Aldous and Diaconis for the purpose of catching up the phase transit of the time to stationarity. To see a definition, let $\mathcal{F}=(\mathcal{X}_n,K_n,\pi_n)_{n=1}^\infty$ be a family of irreducible Markov chains and, for $n\ge 1$, let $d_{n,\text{\tiny TV}}$ and $T_{n,\text{\tiny TV}}$ be the total variation and corresponding mixing time of the $n$th chain in $\mathcal{F}$. Assume that $T_{n,\text{\tiny TV}}(\epsilon_0)\ra\infty$ for some $\epsilon_0\in(0,1)$. The family $\mathcal{F}$ is said to present a cutoff in the total variation if
\begin{equation}\label{eq-defcutoff1}
 \lim_{n\ra\infty}\frac{T_{n,\text{\tiny TV}}(\epsilon)}{T_{n,\text{\tiny TV}}(\delta)}=1,\quad\forall \epsilon,\delta\in(0,1).
\end{equation}
Note that, equivalently, $\mathcal{F}$ has a cutoff in the total variation if there is a sequence of positive reals $(t_n)_{n=1}^\infty$ such that
\begin{equation}\label{eq-defcutoff2}
 \lim_{n\ra\infty}d_{n,\text{\tiny TV}}(\lceil at_n\rceil)=0\quad\forall a>1,\quad\lim_{n\ra\infty}d_{n,\text{\tiny TV}}(\lfloor at_n\rfloor)=1,\quad\forall a\in(0,1).
\end{equation}
From (\ref{eq-defcutoff2}), one can see that the total variation of Markov chains in $\mathcal{F}$ have a phase transition at times $(t_n)_{n=1}^\infty$. When a cutoff exists, the sequence $(t_n)_{n=1}^\infty$, or briefly $t_n$, in (\ref{eq-defcutoff2}) is called a cutoff time and, by (\ref{eq-defcutoff1}), $T_{n,\text{\tiny TV}}(\epsilon)$ can be selected as a cutoff time for any $\epsilon\in(0,1)$.
In the continuous time case, we write $\mathcal{F}_c$ for the family of continuous time chains associated with $\mathcal{F}$ and use $d_{n,\text{\tiny TV}}^{(c)}$ and $T_{n,\text{\tiny TV}}^{(c)}$ to denote the total variation and mixing time of the $n$th chain in $\mathcal{F}_c$. The total variation cutoff of $\mathcal{F}_c$ is defined in the same way through (\ref{eq-defcutoff1}) or (\ref{eq-defcutoff2}) under the replacement of $T_{n,\text{\tiny TV}},d_{n,\text{\tiny TV}}$ with $T_{n,\text{\tiny TV}}^{(c)},d_{n,\text{\tiny TV}}^{(c)}$ and the removal of $\lceil\cdot\rceil,\lfloor\cdot\rfloor$ but without the prerequisite of $T_{n,\text{\tiny TV}}^{(c)}(\epsilon_0)\ra\infty$. The above definitions and discussions are applicable to the Hellinger distance and, in avoidance of any confusion, we use $d_{n,H},d_{n,H}^{(c)}$ and $T_{n,H},T_{n,H}^{(c)}$ to denote the Hellinger distances and mixing times of the $n$th chains in $\mathcal{F},\mathcal{F}_c$.

The study of mixing times and cutoff phenomena for Markov chains was initiated by Aldous, Diaconis and their collaborators in early 1980s. There are many literatures on related topics introduced in the past several decades and we refer readers to \cite{D96} for a concise introduction of cutoff phenomena, to \cite{AF} for classical probabilistic techniques on mixing times, to \cite{D88} for an application of group representation, to \cite{SC04} for random walks on finite groups and to \cite{LPW08} for a rich collection of well-developed techniques.

Based on (\ref{eq-hdtv}) and (\ref{eq-mixhdtv}), we may compare cutoffs in the total variation and in the Hellinger distance as follows.

\begin{prop}\label{t-comp}
Let $\mathcal{F}$ be a family of irreducible Markov chains with countable state spaces and $T_{n,\textnormal{\tiny TV}},T_{n,H}$ be the mixing times as before. Suppose that there is $\epsilon_0\in(0,1)$ such that $T_{n,\textnormal{\tiny TV}}(\epsilon_0)\ra\infty$ or $T_{n,H}(\epsilon_0)\ra\infty$. Then, $\mathcal{F}$ has a cutoff in the total variation if and only if $\mathcal{F}$ has a cutoff in the Hellinger distance. Further, if $\mathcal{F}$ has a cutoff in either the total variation or the Hellinger distance, then $T_{n,\textnormal{\tiny TV}}(\epsilon)/T_{n,H}(\delta)\ra 1$ for all $\epsilon,\delta\in(0,1)$. In the continuous time case, the above conclusion also holds for $\mathcal{F}_c$ without the assumption of $T^{(c)}_{n,\textnormal{\tiny TV}}(\epsilon_0)\ra \infty$ or $T^{(c)}_{n,H}(\epsilon_0)\ra\infty$.
\end{prop}

Proposition \ref{t-comp} is an easy consequence of (\ref{eq-mixhdtv}) and (\ref{eq-defcutoff1}). We refer readers to Proposition \ref{p-comp0} for more comparisons of cutoffs. By Proposition \ref{t-comp}, the total variation cutoff of product chains can be analyzed using their Hellinger distances and the following two examples suitably illustrate this scheme.

\begin{ex}\label{ex-Lacoin}
For $n\ge 1$ and $1\le i\le n$, let $(\mathcal{X}_{n,i},K_{n,i},\pi_{n,i})$ be the Markov chain on $\{0,1,...,2n\}$ with transition matrix given by
\begin{equation}\label{eq-LacoinK}
 \begin{cases}K_{n,i}(j,j+1)=1-a_{n,i},\quad\forall j\notin\{n,2n\},\\
 K_{n,i}(0,0)=K_{n,i}(j,j-1)=a_{n,i},\quad\forall j\notin\{0,2n\},\\
 K_{n,i}(n,n+1)=b_{n,i}n^{-\beta},\,\, K_{n,i}(n,2n)=1-a_{n,i}-b_{n,i}n^{-\beta},\\
 K_{n,i}(2n,n)=c_{n,i},\,\, K_{n,i}(2n,2n)=1-a_{n,i}-c_{n,i},\end{cases}
\end{equation}
where $\beta>0$. See Figure \ref{fig-Lacoin2} for the graph associated with $K_{n,i}$. In the above setting, it is easy to check that $K_{n,i}$ is reversible if and only if
\begin{equation}\label{eq-LacoinRev}
 b_{n,i}c_{n,i}(1-a_{n,i})^{n-1}n^{-\beta}=(1-a_{n,i}-b_{n,i}n^{-\beta})a_{n,i}^n.
\end{equation}
Furthermore, $\pi_{n,i}$ will be concentrated in a neighborhood of $2n$ if the transitions toward $2n$ in $K_{n,i}$ are strong enough.

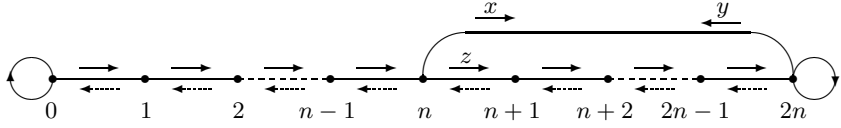
\begin{figure}[htp]

\begin{center}
\begin{picture}(300,60)(-10,-30)

\put(0,0){\put(-8,0){\circle{16}}\put(-16,3){\vector(0,1){0}}}
\put(0,0){\put(288,0){\circle{16}}\put(296,-3){\vector(0,-1){0}}}
\multiput(0,0)(35,0){9}{\put(0,0){\circle*{3}}}
\multiput(0,0)(35,0){2}{\put(0,0){\line(1,0){35}}}
\multiput(70,0)(4,0){9}{\put(0,0){\line(1,0){2}}}
\multiput(105,0)(35,0){3}{\put(0,0){\line(1,0){35}}}
\multiput(210,0)(4,0){9}{\put(0,0){\line(1,0){2}}}
\put(245,0){\line(1,0){35}}
\put(210,0){\oval(140,35)[t]}
\put(-3,-15){{\small$0$}}\put(33,-15){{\small$1$}}\put(68,-15){{\small$2$}}
\put(93,-15){{\small$n-1$}}\put(138,-15){{\small$n$}}
\put(163,-15){{\small$n+1$}}\put(198,-15){{\small$n+2$}}
\put(230,-15){{\small$2n-1$}}\put(275,-15){{\small$2n$}}
\put(163,25){{\small $x$}}\put(251,25){{\small $y$}}
\put(154,6){{\small $z$}}

\multiput(0,0)(35,0){4}{\put(10,4){\line(1,0){12}}
\put(25,4){\vector(1,0){0}}}

\multiput(175,0)(35,0){3}{\put(10,4){\line(1,0){12}}
\put(25,4){\vector(1,0){0}}}

\multiput(0,0)(35,0){8}{\multiput(25,-4)(-2,0){6}{\line(-1,0){1}}
\put(10,-4){\vector(-1,0){0}}}

\put(140,0){\put(10,4){\line(1,0){12}}

\put(25,4){\vector(1,0){0}}}

\put(160,23){\put(0,0){\line(1,0){12}}
\put(15,0){\vector(1,0){0}}}

\put(260,23){\put(0,-2){\line(-1,0){12}}
\put(-15,-2){\vector(-1,0){0}}}

\end{picture}

\caption{The above graph describes the transition matrix in (\ref{eq-LacoinK}). For those
innominate transits, the solid rightward arrows are of probability $1-a_{n,i}$, while the dashed leftward ones are of probability $a_{n,i}$. The nominated transits are respectively $x=1-a_{n,i}-b_{n,i}n^{-\beta}$, $y=c_{n,i}$ and $z=b_{n,i}n^{-\beta}$, while the loops are set to make $K_{n,i}$ stochastic. \label{fig-Lacoin2}}

\end{center}
\end{figure}

This model was first introduced by Lacoin in \cite{L15} for the purpose of illustrating product chains without cutoffs in the total variation and separation. Here, we refine partial results in \cite{L15} by showing the sensitivity of cutoffs with respect to the transition probabilities in $K_{n,i}$. In Lemma \ref{l-Lacoin1}, we provides sharp bounds on the Hellinger distance of the product chain of $(\mathcal{X}_{n,i},K_{n,i},\pi_{n,i})_{i=1}^n$. As a consequence, we obtain simple criteria to determine the total variation cutoff in Proposition \ref{p-Lacoin2} and Corollary \ref{c-Lacoin}. The following proposition treats a special case of (\ref{eq-LacoinK}) and is a
consequence of Proposition \ref{p-Lacoin2} and Corollary \ref{c-Lacoin}. Its proof is placed in the appendix for completion.

\begin{prop}\label{p-Lacoin3}
Let $p_{n,i}>0$, $(\mathcal{X}_{n,i},K_{n,i},\pi_{n,i})$ be the Markov chain satisfying \textnormal{(\ref{eq-LacoinK})-(\ref{eq-LacoinRev})} and $q_n=p_{n,1}+\cdots+p_{n,n}$. Consider the family $\mathcal{G}=(\mathcal{X}_n,K_n,\pi_n)_{n=1}^\infty$, where $(\mathcal{X}_n,K_n,\pi_n)$ is the product chain of $(\mathcal{X}_{n,i},K_{n,i},\pi_{n,i})_{i=1}^n$ according to the probability vector $(p_{n,1}/q_n,...,p_{n,n}/q_n)$. Suppose there is $C>1$ such that
\begin{equation}\label{eq-Lacoin4}
 \sum_{i=1}^na_{n,i}\le Cn^{-\beta-1},\quad C^{-1}\le b_{n,i}\le C,\quad\forall 1\le i\le n,\,n\ge 1.
\end{equation}
\begin{itemize}
\item[(1)] For $p_{n,i}=1+2^{i-n}$, $\mathcal{G}_c$ has a total variation cutoff if and only if $\beta\ne 1$. Further, if $\beta\in(0,1)$, then the cutoff time is $2n^2$; if $\beta\in(1,\infty)$, then the cutoff time is $n^2$.

\item[(2)] For $p_{n,i}=1+(i/n)^\alpha$ with $\alpha>0$, $\mathcal{G}_c$ has a total variation cutoff if and only if $\beta\ne 1$. Further, if $\beta\in(0,1)$, then the cutoff time is $2[(\alpha+2)/(\alpha+1)]n^2$; if $\beta\in(1,\infty)$, then the cutoff time is $[(\alpha+2)/(\alpha+1)]n^2$.

\item[(3)] For $p_{n,i}=1+\log i/\log n$, $\mathcal{G}_c$ has a total variation cutoff with cutoff time $4n^2/(1+\min\{\beta,1\})$ for all $\beta>0$.
\end{itemize}
\end{prop}

In \cite{L15}, Lacoin creates the continuous time Markov chains without cutoff by directly assigning their $Q$-matrices. To our setting, the transition matrices have $\beta=1$ and, roughly, $a_{n,i}=2^{-n^2}$, $b_{n,i}=1$ and $c_{n,i}=n^{-1}2^{-n^3}$. It is easy to check that (\ref{eq-Lacoin4}) is satisfied and, by Proposition \ref{p-Lacoin3}, no cutoff exists in the total variation.

\end{ex}

Next, we consider some specific type of product chains and do its framework on the comparison of cutoffs between product chains and original chains. In detail, let $\mathcal{F}=(\mathcal{X}_n,K_n,\pi_n)_{n=1}^\infty$ be a family of Markov chains and $\mathcal{P}=(p_n)_{n=1}^\infty$ be a sequence of positive reals. For $n\ge 1$, let $q_n=\sum_{i=1}^np_i$ and $(\mathcal{Y}_n,L_n,\nu_n)$ be the product of $(\mathcal{X}_i,K_i,\pi_i)_{i=1}^n$ according to the probability vector $(p_1/q_n,...,p_n/q_n)$. We write $\mathcal{F}^\mathcal{P}$ for the family $(\mathcal{Y}_n,L_n,\nu_n)_{n=1}^\infty$ and write $\mathcal{F}^{\mathcal{P}}_c$ for the family of continuous time chains associated with $\mathcal{F}^{\mathcal{P}}$. When we say a subfamily of $\mathcal{F}$, we mean $(\mathcal{X}_{\xi_n},K_{\xi_n},\pi_{\xi_n})_{n=1}^\infty$, where $(\xi_n)_{n=1}^\infty$ is an increasing sequence of positive integers. The following theorem provides criteria on the cutoff of $\mathcal{F}^\mathcal{P}_c$ with specific $\mathcal{P}$.

\begin{thm}\label{t-main}
Let $\mathcal{F}^\mathcal{P}$ be the family introduced above, $\epsilon_n$ be a sequence satisfying $0<\inf_n\epsilon_n\le\sup_n\epsilon_n<1/2$ and set
\[
 D_n:=\log\frac{T_{n,\textnormal{TV}}^{(c)}(\epsilon_n)}{p_n}=A_nn+B_n+C_n.
\]
Assume that:
\begin{itemize}
\item[(I)] Either $0<A_n\le A_{n+1}$ for all $n$ or $n|A_n-A|$ is bounded for some $A>0$.

\item[(II)] $B_n$ is nondecreasing, $C_n$ is bounded and $D_n$ is nondecreasing for $n$ large enough.
\end{itemize}
In the total variation:
\begin{itemize}
\item[(1)] If $\mathcal{F}_c$ has a cutoff with cutoff time $t_n$, then $\mathcal{F}^\mathcal{P}_c$ has a cutoff with cutoff time $(p_1+\cdots+p_n)t_n/p_n$.

\item[(2)] If no subfamily of $\mathcal{F}_c$ has a cutoff, then $\mathcal{F}^\mathcal{P}_c$ has no cutoff.
\end{itemize}

The above conclusions also hold in the Hellinger distance if $\sup_n\epsilon_n<1/4$ is assumed further and $T_{n,\textnormal{\tiny TV}}^{(c)}$ is replaced by $T_{n,H}^{(c)}$.
\end{thm}

A general version of Theorem \ref{t-main} is discussed in Subsection \ref{ss-sar} and readers are referred to Theorem \ref{t-prodseqtv} for more details. To see a practical application, we consider products of random walks on finite cycles.

\begin{prop}\label{p-psrw}
Refer to the family $\mathcal{F}^\mathcal{P}$ in Theorem \ref{t-main} and let
$\mathcal{X}_n=\mathbb{Z}_{n+1}$, $K_n(x,y)=1/2$ for $|x-y|=1$ and $p_n=n^2\exp\{-n^{\gamma}\}$ with $\gamma>0$. If $\gamma>1$, then $\mathcal{F}^\mathcal{P}_c$ has no cutoff in the total variation.
\end{prop}

It is well-known that the total variation mixing time of the $n$th chain in $\mathcal{F}_c$ has order $n^2$. Noting this, Proposition \ref{p-psrw} is a consequence of Theorem \ref{t-main} and the observation of $(n+1)^\gamma-n^\gamma\ge n^{\gamma-1}$. In the forthcoming paper \cite{CK17}, we have more advanced analysis on the cutoff of  product chains for finite groups with moderate growth, which is a generalization of Proposition \ref{p-psrw}. It is shown in \cite{CK17} that, when the pre-cutoff (a concept weaker than the cutoff) is considered, the family $\mathcal{F}^\mathcal{P}_c$ in Proposition \ref{p-psrw} presents a pre-cutoff in the total variation for $\gamma\in(0,1)$, but
does not for $\gamma\ge 1$. This means that Theorem \ref{t-main} could be sharp in judging cutoffs.

As is revealed in Theorem \ref{t-main}, the cutoffs for $\mathcal{F}_c$ and $\mathcal{F}_c^{\mathcal{P}}$ are consistent under some mild conditions. However, this can fail in general and we provide counterexamples in Subsection \ref{ss-noncnsst} to highlight the observation of the following theorem.

\begin{thm}\label{thm:counterexs}
None of cutoffs for $\mathcal{F}_c$ or $\mathcal{F}_c^\mathcal{P}$ implies the other.
\end{thm}

The remaining sections of this article are organized in the following way. In Section \ref{s-cc}, a comparison between the total variation and the Hellinger distance is introduced to relate the cutoff in one measurement with the cutoff in the other, where Proposition \ref{t-comp} is a typical result in the framework. In Section \ref{s-dpc}, we consider product chains in the continuous time case and, based on (\ref{eq-prodcts}), create a list of bounds on their mixing times. In Section \ref{s-cpc}, the combination of the comparison technique and the bounds for product chains leads to a series of criteria on the existence of cutoffs and related materials. In Section \ref{s-examples}, we consider the family in Theorem \ref{t-main} and determine its cutoff to some extent. For illustration, we consider products of two-state chains and a general family of chains in Proposition \ref{p-Lacoin3}. Besides, two examples are introduced to reveal the non-consistency of cutoffs, which provide the proof of Theorem \ref{thm:counterexs}.
We would like to emphasize that those heuristic examples in Section \ref{s-examples} are helpful to understand the theoretic development in this paper though the discussion within the section and the auxiliary proofs relegated in the appendix occupy a significantly large part.

We end the introduction by quoting a list of mathematical notations to be used throughout this article. Let $x,y\in\mathbb{R}$ and $a_n,b_n$ be sequences of positive reals. We write $x\vee y$ and $x\wedge y$ for the maximum and minimum of $x$ and $y$. When $a_n/b_n$ is bounded, we write $a_n=O(b_n)$; when $a_n/b_n\ra 0$, we write $a_n=o(b_n)$. In the case of $a_n=O(b_n)$ and $b_n=O(a_n)$, we simply say $a_n\asymp b_n$. If $a_n/b_n\ra 1$, we write $a_n\sim b_n$. When writing $O(a_n)$ and $o(b_n)$ as a single term, we mean sequences, $c_n$ and $d_n$, satisfying $|c_n/a_n|=O(1)$ and $|d_n/b_n|=o(1)$ respectively.

\section{Comparison of cutoffs}\label{s-cc}

In this section, we consider the total variation and the Hellinger distance in a more general setting and provides a comparison of mixing times in both measurements.

\subsection{Comparisons of the total variation and Hellinger distance}

Let $\mathcal{X}$ be a set equipped with $\sigma$-field $\mathcal{A}$. For any two probabilities $\mu,\nu$ on $(\mathcal{X},\mathcal{A})$, the total variation and the Hellinger distance are defined by
\begin{equation}\label{eq-tv0}
 \|\mu-\nu\|_{\text{\tiny TV}}:=\sup_{A\in\mathcal{A}}\{\mu(A)-\nu(A)\},
\end{equation}
and
\begin{equation}\label{eq-hd0}
 \|\mu-\nu\|_H:=\sqrt{\frac{1}{2}\int_\mathcal{X}\left(\sqrt{\frac{d\mu}{d\lambda}}
 -\sqrt{\frac{d\nu}{d\lambda}}\right)^2d\lambda}=\sqrt{1-\int_\mathcal{X}
 \sqrt{\frac{d\mu}{d\lambda}\frac{d\nu}{d\lambda}}d\lambda},
\end{equation}
where $\lambda$ is a probability on $(\mathcal{X},\mathcal{A})$ such that $d\mu/d\lambda$ and $d\nu/d\lambda$ exist. The total variation is clearly well-defined in (\ref{eq-tv0}), while the Hellinger distance requires the existence and independence of $\lambda$ in (\ref{eq-hd0}). To see (\ref{eq-hd0}) is well-defined, let $(P,N)$ be a Hahn decomposition of $\mu-\nu$ satisfying $\mu(P)\ge \nu(P)$, $\mu(N)\le \nu(N)$ and define $\pi$ by
\begin{equation}\label{eq-pi}
 \pi(A)=\mu(P\cap A)+\nu(N\cap A),\quad\forall A\in\mathcal{A}.
\end{equation}
By setting $c=\mu(P)+\nu(N)$, it is easy to see that $c^{-1}\pi$ is a probability and $\mu,\nu$ are absolutely continuous with respect to $\pi$. This provides a candidate of $\lambda$. Next, let $f,g$ be Radon derivatives of $\mu,\nu$ with respective to $\pi$ and let $\lambda$ be a probability with respect to which $\mu$ and $\nu$ are absolutely continuous. Obviously, $\pi$ is absolutely continuous with respect to $\lambda$ since $\pi\le\mu+\nu$. As a consequence, (\ref{eq-hd0}) can be rewritten as
\begin{equation}\label{eq-hd}
 1-\|\mu-\nu\|_H^2=\int_\mathcal{X}\sqrt{fg}d\pi.
\end{equation}
This proves the independence of $\lambda$ in (\ref{eq-hd0}).

The following lemma is known (see for instance \cite[Equation (25) on p.365]{Sh96}) and we give its proof for reader's convenience.

\begin{lem}\label{l-comp}
For any two probabilities $\mu,\nu$, one has
\[
 1-\sqrt{1-\|\mu-\nu\|_{\textnormal{\tiny TV}}^2}\le \|\mu-\nu\|_H^2\le \|\mu-\nu\|_{\textnormal{\tiny TV}}.
\]
\end{lem}

\begin{rem}\label{r-hdtv}
The first inequality in Lemma \ref{l-comp} implies
\[
 \|\mu-\nu\|_{\text{\tiny TV}}\le \|\mu-\nu\|_H\sqrt{2-\|\mu-\nu\|_H^2}
 \le\sqrt{2}\|\mu-\nu\|_H,
\]
while the fact of $\|\mu-\nu\|_{\text{\tiny TV}}\le\sqrt{2}\|\mu-\nu\|_H$ is also derived in \cite{LPW08,SC97}.
\end{rem}

\begin{proof}
Let $f,g$ be as before. Observe that
\[
 f=\begin{cases}1&\text{on }P\\\frac{d\mu|_N}{d\nu|_N}&\text{on }N\end{cases},\quad g=\begin{cases}\frac{d\nu|_P}{d\mu|_P}&\text{on }P\\1&\text{on }N\end{cases}.
\]
where $\mu|_A$ denotes the restriction of $\mu$ to set $A$. This implies
\begin{equation}\label{eq-tv1}
 1-\|\mu-\nu\|_{\text{\tiny TV}}=\mu(N)+\nu(P)=\int_\mathcal{X}fgd\pi.
\end{equation}
Besides, by the definition in (\ref{eq-tv0}) and the setting in (\ref{eq-pi}), it is easy to see that
\begin{equation}\label{eq-tv2}
 1+\|\mu-\nu\|_{\text{\tiny TV}}=\mu(P)+\nu(N)=\pi(\mathcal{X}).
\end{equation}
Since $f,g$ are bounded by $1$, one has $0\le fg\le 1$. By (\ref{eq-hd}) and (\ref{eq-tv1}), this yields $1-\|\mu-\nu\|_H^2\ge 1-\|\mu-\nu\|_{\text{\tiny TV}}$ and
\[
 1-\|\mu-\nu\|_H^2\le\sqrt{\pi(\mathcal{X})\int_{\mathcal{X}}fgd\pi}
 =\sqrt{1-\|\mu-\nu\|^2_{\text{\tiny TV}}},
\]
where the first inequality is exactly the Cauchy-Schwarz inequality and the last equality applies (\ref{eq-tv2}).
\end{proof}

To see an application of Lemma \ref{l-comp}, we consider products of probabilities.

\begin{prop}\label{p-comp}
Fix $n\in\mathbb{N}$. For $1\le i\le n$, let $\mu_i,\nu_i$ be probabilities on the same measurable space and set $\mu=\mu_1\times\cdots\times\mu_n$ and $\nu=\nu_1\times\cdots\times\nu_n$. In the Hellinger distance, one has
\begin{equation}\label{eq-phd}
 \|\mu-\nu\|_H^2=1-\prod_{i=1}^n(1-\|\mu_i-\nu_i\|_H^2)\ge\max_{1\le i\le n}\|\mu_i-\nu_i\|_H^2.
\end{equation}
In the total variation, one has $\|\mu-\nu\|_{\textnormal{\tiny TV}}\ge\max\{\|\mu_i-\nu_i\|_{\textnormal{\tiny TV}}:1\le i\le n\}$ and
\begin{equation}\label{eq-ptv}
 1-\prod_{i=1}^n\left(1-\|\mu_i-\nu_i\|^2_{\textnormal{\tiny TV}}\right)^{1/2}\le \|\mu-\nu\|_{\textnormal{\tiny TV}}\le 1-\prod_{i=1}^n(1-\|\mu_i-\nu_i\|_{\textnormal{\tiny TV}}).
\end{equation}
\end{prop}

The equality in (\ref{eq-phd}) was early introduced in \cite{LPW08} (see Exercise 20.5) and we display a proof in this article for completion.

\begin{proof}[Proof of Proposition \ref{p-comp}]
For convenience, let $(\mathcal{X}_i,\mathcal{A}_i)$ be the measurable space on which $\mu_i,\nu_i$ are defined and set $\mathcal{X}=\prod_{i=1}^n\mathcal{X}_i$ and $\mathcal{A}=\bigotimes_{i=1}^n\mathcal{A}_i$. We first prove the equality in (\ref{eq-phd}). For $1\le i\le n$, let $(P_i,N_i)$ be a Hahn decomposition of $\mu_i-\nu_i$ such that $\mu_i(P_i)\ge \nu_i(P_i)$ and $\mu_i(N_i)\le \nu_i(N_i)$. By (\ref{eq-hd}), one has
\[
 1-\|\mu_i-\nu_i\|_H^2=\int_{\mathcal{X}_i}\sqrt{\frac{d\mu_i}{d\pi_i}
 \frac{d\nu_i}{d\pi_i}}d\pi_i,
\]
where $\pi_i(A)=\mu_i(P_i\cap A)+\nu_i(N_i\cap A)$ for $A\in\mathcal{A}_i$. Set $\pi=\pi_1\times\cdots\times\pi_n$. Clearly, $\mu$ and $\nu$ are absolutely continuous with respect to $\pi$ and
\[
 \frac{d\mu}{d\pi}(x_1,...,x_n)=\prod_{i=1}^n\frac{d\mu_i}{d\pi_i}(x_i),\quad
 \frac{d\nu}{d\pi}(x_1,...,x_n)=\prod_{i=1}^n\frac{d\nu_i}{d\pi_i}(x_i).
\]
As a result, (\ref{eq-hd0}) implies
\[
 1-\|\mu-\nu\|_H^2=\int_\mathcal{X}\sqrt{\frac{d\mu}{d\pi}\frac{d\nu}{d\pi}}d\pi
 =\prod_{i=1}^n\int_{\mathcal{X}_i}\sqrt{\frac{d\mu_i}{d\pi_i}\frac{d\nu_i}{d\pi_i}}
 d\pi_i=\prod_{i=1}^n(1-\|\mu_i-\nu_i\|_H^2).
\]
The inequality in (\ref{eq-phd}) is obvious and skipped.

Next, we show (\ref{eq-ptv}). Note that the first inequality follows immediately from (\ref{eq-phd}) and Lemma \ref{l-comp}. To see the second inequality, we set $\hat{\pi}_i(A)=\mu_i(N_i\cap A)+\nu_i(P_i\cap A)$ for $A\in\mathcal{A}_i$, $\hat{\pi}=\hat{\pi}_1\times\cdots\times\hat{\pi}_n$ and let $(P,N)$ be a Hahn decomposition of $\mu-\nu$ satisfying $\mu(P)\ge \nu(P)$ and $\mu(N)\le \nu(N)$.
As $\hat{\pi}(\mathcal{X})=\prod_{i=1}^n(1-\|\mu_i-\nu_i\|_{\text{\tiny TV}})$ and $1-\|\mu-\nu\|_{\text{\tiny TV}}=\mu(N)+\nu(P)$, the second inequality in (\ref{eq-ptv}) becomes
\begin{equation}\label{eq-pn}
 \hat{\pi}(\mathcal{X})\le\mu(N)+\nu(P).
\end{equation}
Observe that, on $D=\prod_{i=1}^nD_i$ with $D_i\in\{P_i,N_i\}$,
\[
 \frac{d\mu}{d\pi}(x_1,...,x_n)=\prod_{i:D_i=N_i}\frac{d\mu_i}{d\nu_i}(x_i),\quad
 \frac{d\nu}{d\pi}(x_1,...,x_n)=\prod_{i:D_i=P_i}\frac{d\nu_i}{d\mu_i}(x_i)
\]
and
\[
 \frac{d\hat{\pi}}{d\pi}(x_1,...,x_n)=\prod_{i:D_i=N_i}\frac{d\mu_i}{d\nu_i}(x_i)\times
 \prod_{i:D_i=P_i}\frac{d\nu_i}{d\mu_i}(x_i).
\]
As $d\mu_i/d\nu_i\le 1$ on $N_i$ and $d\nu_i/d\mu_i\le 1$ on $P_i$, the above identities imply
\[
 \frac{d\hat{\pi}}{d\pi}=\frac{d\mu}{d\pi}\frac{d\nu}{d\pi}\le\frac{d\mu}{d\pi}
 \wedge\frac{d\nu}{d\pi}=\frac{d\mu}{d\pi}\mathbf{1}_N+\frac{d\nu}{d\pi}\mathbf{1}_P,
\]
which leads to (\ref{eq-pn}).

To prove the other lower bound of the total variation, let $A_i=\{x\in\mathcal{X}_i|\mu_i(x)\ge\nu_i(x)\}$ and $B_i=\{x=(x_1,...,x_n)\in\mathcal{X}|x_i\in A_i\}$. Then, one has
\[
 \|\mu-\nu\|_{\textnormal{\tiny TV}}\ge\mu(B_i)-\nu(B_i)=\mu_i(A_i)-\nu_i(A_i)=\|\mu_i-\nu_i\|_{\textnormal{\tiny TV}},\quad\forall 1\le i\le n.
\]
\end{proof}

\subsection{Mixing times of Markov chains and their comparisons}

Let $(\mathcal{X},K,\pi)$ be an irreducible Markov chain on a countable set $\mathcal{X}$ with transition matrix $K$ and stationary distribution $\pi$ and let $(\mathcal{X},H_t,\pi)$ be the continuous time Markov chain associated with $(\mathcal{X},K,\pi)$, where $H_t=e^{-t(I-K)}$. If those Markov chains have $\mu$ as the initial distribution, we write $(\mu,\mathcal{X},K,\pi)$ and $(\mu,\mathcal{X},H_t,\pi)$ instead. When $\mu=\delta_x$, a probability concentrated at state $x$, we simply write $(x,\mathcal{X},K,\pi)$ and $(x,\mathcal{X},H_t,\pi)$.

Referring to (\ref{eq-tv0})-(\ref{eq-hd0}), we define the total variation and the Hellinger distance of $(\mu,\mathcal{X},K,\pi)$ by
\begin{equation}\label{eq-hdtvdf}
 d_{\text{\tiny TV}}(\mu,m)=\|\mu K^m-\pi\|_{\text{\tiny TV}},\quad d_H(\mu,m)=\|\mu K^m-\pi\|_H,
\end{equation}
and define those of $(\mathcal{X},K,\pi)$ by
\begin{equation}\label{eq-maxhdtvdf}
 d_{\text{\tiny TV}}(m)=\sup_{\mu}d_{\text{\tiny TV}}(\mu,m),\quad
 d_H(m)=\sup_\mu d_H(\mu,m).
\end{equation}
For simplicity, we also call the distances in (\ref{eq-maxhdtvdf}) the maximum total variation and the maximum Hellinger distance. The mixing times associated with $d_{\text{\tiny TV}}$ and $d_H$ are set to be
\[
 T_{\text{\tiny TV}}(\mu,\epsilon):=\inf\{m\ge 0|d_{\text{\tiny TV}}(\mu,m)\le\epsilon\},\quad T_{\text{\tiny TV}}(\epsilon):=\inf\{m\ge 0|d_{\text{\tiny TV}}(m)\le\epsilon\},
\]
and
\[
 T_H(\mu,\epsilon):=\inf\{m\ge 0|d_H(\mu,m)\le\epsilon\},\quad T_H(\epsilon):=\inf\{m\ge 0|d_H(m)\le\epsilon\}.
\]
When $\mu=\delta_x$, we write $d_{\text{\tiny TV}}(x,m)$, $d_H(x,m),T_{\text{\tiny TV}}(x,\epsilon)$ and $T_H(x,\epsilon)$ for short. Concerning the continuous time case, we change $K^m$ into $H_t$ in the above definitions and, to avoid confusion, replace $d_{\text{\tiny TV}},T_{\text{\tiny TV}},d_H,T_H$ with $d^{(c)}_{\text{\tiny TV}},T^{(c)}_{\text{\tiny TV}},d_H^{(c)},T_H^{(c)}$. Note that the total variation, the Hellinger distance and their corresponding mixing times are non-increasing.

As a result of Lemma \ref{l-comp}, we provide in the following lemma a comparison between the total variation and the Hellinger distance. It is remarkable that the two distances
are simultaneously close to $0$ and $1$, which is useful to identify cutoffs, introduced in the next subsection, in either measurements.
\begin{lem}\label{l-comp2}
Let $d_{\textnormal{\tiny TV}}(\mu,\cdot),d_H(\mu,\cdot)$ be distances in \textnormal{(\ref{eq-hdtvdf})} and $T_{\textnormal{\tiny TV}}(\mu,\cdot),T_{H}(\mu,\cdot)$ be their corresponding mixing times. Then, one has
\begin{equation}\label{eq-hdtv0}
 1-\sqrt{1-d_{\textnormal{\tiny TV}}^2(\mu,m)}\le d_H^2(\mu,m)\le d_{\textnormal{\tiny TV}}(\mu,m),\quad\forall m\ge 0,
\end{equation}
and
\begin{equation}\label{eq-mixhdtv0}
 T_{\textnormal{\tiny TV}}(\mu,\epsilon\sqrt{2-\epsilon^2})\le T_H(\mu,\epsilon)\le T_{\textnormal{\tiny TV}}(\mu,\epsilon^2),\quad\forall \epsilon\in(0,1).
\end{equation}
The above inequalities also hold in the distances of \textnormal{(\ref{eq-maxhdtvdf})} and in the continuous time case.
\end{lem}

Concerning (\ref{eq-hdtv0}), it's interesting to explore whether there is a universal constant $C>0$ independent of the Markov chain such that
\[
  1-\sqrt{1-d_{\text{\tiny TV}}^2(\mu,m)}\ge Cd_H^2(\mu,m),\quad\forall m\ge 0,
\]
or
\[
 d_{\text{\tiny TV}}(\mu,m)\le Cd_H^2(\mu,m),\quad\forall m\ge 0.
\]
In the following example, we demonstrate that none of the above inequalities can hold.

\begin{ex}
Let $(\mathcal{X},K,\pi)$ be a Markov chain with
\begin{equation}\label{eq-2sc}
 \mathcal{X}=\{0,1\},\quad K=\left(\begin{array}{cc}1-\alpha&\alpha\\\beta&1-\beta\end{array}\right),\quad
 \pi=\left(\frac{\beta}{\alpha+\beta},\frac{\alpha}{\alpha+\beta}\right)
\end{equation}
It is easy to see that $K$ is reversible and to show that
\begin{equation}\label{eq-2pK}
 K^m(0,0)=\frac{\beta}{\alpha+\beta}+\frac{\alpha}{\alpha+\beta}(1-\alpha-\beta)^m.
\end{equation}
This implies
\[
 d_H(0,m)^2=1-\frac{\beta}{\alpha+\beta}\sqrt{1+\frac{\alpha}{\beta}(1-\alpha-\beta)^m}
 -\frac{\alpha}{\alpha+\beta}\sqrt{1-(1-\alpha-\beta)^m}
\]
and
\[
 d_{\text{\tiny TV}}(0,m)=\frac{\alpha}{\alpha+\beta}(1-\alpha-\beta)^m.
\]
By the fact of $\sqrt{1+u}=1+u/2-u^2/8+O(u^3)$ as $u\ra 0$, one may derive
\[
 d_H(0,m)^2\sim \frac{\alpha(1-\alpha-\beta)^{2m}}{8\beta},\quad
 1-\sqrt{1-d_{\text{\tiny TV}}^2(0,m)}\sim\frac{\alpha^2(1-\alpha-\beta)^{2m}}{2(\alpha+\beta)^2},
\]
as $m\ra\infty$. As a consequence, we obtain
\begin{equation}\label{eq-hdtvcomp}
 \frac{d_{\text{\tiny TV}}(0,m)}{d_H^2(0,m)}\sim\frac{8\beta(1-\alpha-\beta)^{-m}}{\alpha+\beta},\quad
 \frac{1-\sqrt{1-d_{\text{\tiny TV}}^2(0,m)}}{d_H^2(0,m)}\sim\frac{4\alpha\beta}{(\alpha+\beta)^2},
\end{equation}
as $m\ra\infty$. Clearly, the former sequence in (\ref{eq-hdtvcomp}) tends to infinity, while the limit of the latter sequence can be arbitrarily close to zero when $\alpha\beta$ is small.
\end{ex}

\subsection{Cutoffs for Markov chains and their comparisons}

When discussing cutoffs, we refer to a family of Markov chains. To see a precise definition, we introduce the following notations. Let $\mathcal{F}=(\mathcal{X}_n,K_n,\pi_n)_{n=1}^\infty$ be a family of irreducible Markov chain and write $\mathcal{F}_c$ for $(\mathcal{X}_n,H_{n,t},\pi_n)_{n=1}^\infty$, where $H_{n,t}=e^{-t(I-K_n)}$. Here, we call $\mathcal{F}_c$ the family of continuous time Markov chains associated with $\mathcal{F}$. When dealing with $(\mu_n,\mathcal{X}_n,K_n,\pi_n)_{n=1}^\infty$, we call it a family of irreducible Markov chains with initial distributions $(\mu_n)_{n=1}^\infty$. For $n\ge 1$, we write $d_{n,\text{\tiny TV}}$ and $d_{n,H}$ for the total variation and the Hellinger distance of the $n$th chain in $\mathcal{F}$ and let $T_{n,\text{\tiny TV}}$ and $T_{n,H}$ be the corresponding mixing times.

\begin{defn}\label{d-cutoff}
A family $\mathcal{F}$ of irreducible Markov chains with initial distributions $(\mu_n)_{n=1}^\infty$ is said to present
\begin{itemize}
\item[(1)] a cutoff in the total variation if there is $t_n>0$ such that
\[
 \lim_{n\ra\infty}d_{n,\text{\tiny TV}}(\mu_n,\lceil at_n\rceil)=0,\,\,\forall a>1,\quad \lim_{n\ra\infty}d_{n,\text{\tiny TV}}(\mu_n,\lfloor at_n\rfloor)=1,\,\,\forall 0<a<1.
\]

\item[(2)] a $(t_n,b_n)$ cutoff in the total variation if $t_n>0$, $b_n>0$, $b_n=o(t_n)$ and
\[
 \lim_{c\ra\infty}\overline{f}(c)=0,\quad \lim_{c\ra-\infty}\underline{f}(c)=1,
\]
where
\[
 \overline{f}(c):=\limsup_{n\ra\infty}d_{n,\text{\tiny TV}}(\mu_n,\lceil t_n+cb_n\rceil),\quad \underline{f}(c):=\liminf_{n\ra\infty}d_{n,\text{\tiny TV}}(\mu_n,\lfloor t_n+cb_n\rfloor).
\]
\end{itemize}
In the above setting, $t_n$ is called a cutoff time, $b_n$ is called a cutoff window corresponding to $t_n$ and $\overline{f},\underline{f}$ are called the $(t_n,b_n)$ cutoff profiles.
\end{defn}

Referring to Definition \ref{d-cutoff}, the cutoff in the Hellinger distance is defined by replacing $d_{n,\text{\tiny TV}}$ with $d_{n,H}$. If the initial distributions are not specified, the cutoff is understood in the distance of (\ref{eq-maxhdtvdf}) and defined by replacing $d_{n,\text{\tiny TV}}(\mu_n,\cdot),d_{n,H}(\mu_n,\cdot)$ with $d_{n,\text{\tiny TV}}(\cdot),d_{n,H}(\cdot)$. In the continuous time case, the cutoff of $\mathcal{F}_c$ is defined by using $d_{n,\text{\tiny TV}}^{(c)},d_{n,H}^{(c)}$ instead and removing $\lceil\cdot\rceil,\lfloor\cdot\rfloor$.

The following lemma provides another variant of cutoffs using the mixing times.

\begin{lem}{\rm (\cite[Propositions 2.3-2.4]{CSal08})}
\label{l-cutoff}
Let $\mathcal{F}$ be a family of irreducible Markov chains with initial distributions $(\mu_n)_{n=1}^\infty$. Suppose $T_{n,\textnormal{\tiny TV}}(\mu_n,\epsilon_0)\ra\infty$ for some $\epsilon_0\in(0,1)$.
\begin{itemize}
\item[(1)] $\mathcal{F}$ has a cutoff in the total variation if and only if
\[
 T_{n,\textnormal{\tiny TV}}(\mu_n,\epsilon)\sim T_{n,\textnormal{\tiny TV}}(\mu_n,1-\epsilon),\quad \forall \epsilon\in(0,1).
\]
In particular, if $\mathcal{F}$ has cutoff time $t_n$, then $T_{n,\textnormal{\tiny TV}}(\mu_n,\epsilon)\sim t_n$ for $\epsilon\in(0,1)$.

\item[(2)] Assume that $\inf_nb_n>0$. Then, $\mathcal{F}$ has a $(t_n,b_n)$ cutoff in the total variation if and only if $b_n=o(t_n)$ and
\[
 |T_{n,\textnormal{\tiny TV}}(\mu_n,\epsilon)-t_n|=O(b_n), \quad\forall \epsilon\in(0,1).
\]
In particular, for $\epsilon_1\in(0,1)$ and $t_n=T_{n,\textnormal{\tiny TV}}(\mu_n,\epsilon_1)$, $\mathcal{F}$ has a $(t_n,b_n)$ cutoff in the total variation if and only if $b_n=o(t_n)$ and
\[
  |T_{n,\textnormal{\tiny TV}}(\mu_n,\epsilon)-T_{n,\textnormal{\tiny TV}}(\mu_n,1-\epsilon)|=O(b_n),\quad\forall\epsilon\in(0,1).
\]
\end{itemize}

The above statements are also valid for cutoffs in the Hellinger distance and in the distances of \textnormal{(\ref{eq-maxhdtvdf})}, and for $\mathcal{F}_c$, where the assumptions of $T^{(c)}_{n,\textnormal{\tiny TV}}(\mu_n,\epsilon_0)\ra\infty$ and $\inf_nb_n>0$ are not required in the continuous time case.
\end{lem}

The following proposition provides a comparison of cutoffs in the total variation and the Hellinger distance.

\begin{prop}\label{p-comp0}
Let $\mathcal{F}$ be a family of irreducible Markov chains with initial distributions $(\mu_n)_{n=1}^\infty$.
\begin{itemize}
\item[(1)] $\mathcal{F}$ has a cutoff in the total variation with cutoff time $t_n$ if and only if $\mathcal{F}$ has a cutoff in the Hellinger distance with cutoff time $t_n$. Further, if $t_n\ra\infty$, then $T_{n,\textnormal{\tiny TV}}(\mu_n,\epsilon)\sim T_{n,H}(\mu_n,\delta)$ for all $\epsilon,\delta\in(0,1)$.

\item[(2)] $\mathcal{F}$ has a $(t_n,b_n)$ cutoff in the total variation if and only if $\mathcal{F}$ has a $(t_n,b_n)$ cutoff in the Hellinger distance. Further, if $\inf_nb_n>0$, then $|T_{n,\textnormal{\tiny TV}}(\mu_n,\epsilon)-T_{n,H}(\mu_n,\delta)|=O(b_n)$ for all $\epsilon,\delta\in(0,1)$.

\item[(3)] Assume that $\mathcal{F}$ has a $(t_n,b_n)$ cutoff in the total variation and the Hellinger distance and let $\overline{f}_{\textnormal{\tiny TV}},\underline{f}_{\textnormal{\tiny TV}}$ and $\overline{f}_H,\underline{f}_H$ be $(t_n,b_n)$ cutoff profiles in respective distances. Then, one has
\[
 1-\sqrt{1-\overline{f}_{\textnormal{\tiny TV}}^2(c)}\le\overline{f}_H^2(c)\le\overline{f}_{\textnormal{\tiny TV}}(c),\quad
 1-\sqrt{1-\underline{f}_{\textnormal{\tiny TV}}^2(c)}\le\underline{f}_H^2(c)\le\underline{f}_{\textnormal{\tiny TV}}(c)
\]
\end{itemize}
The above also holds in the distance of \textnormal{(\ref{eq-maxhdtvdf})} and in the continuous time case, where $t_n\ra\infty$ and $\inf_nb_n>0$ are not required for $\mathcal{F}_c$.
\end{prop}

\begin{proof}
The proof follows immediately from Lemmas \ref{l-comp2}-\ref{l-cutoff} and is skipped.
\end{proof}

\subsection{Comparisons of cutoffs: Continuous time vs. Discrete time}

In \cite{CSal13-1}, Chen and Saloff-Coste compare the total variation cutoffs between the continuous time chains and lazy discrete time chains, while the next
proposition also provides a similar comparison of cutoffs in the Hellinger distance.

\begin{prop}\label{p-comp2}
Let $\mathcal{F}=(\mu_n,\mathcal{X}_n,K_n,\pi_n)_{n=1}^\infty$ be a family of irreducible Markov chains and $\mathcal{F}_c$ be the family of continuous time chains associated with $\mathcal{F}$. For any sequence $\theta=(\theta_n)_{n=1}^\infty$ in $(0,1)$, set $\mathcal{F}_\theta=(\mu_n,\mathcal{X}_n,K_{n,\theta_n},\pi_n)_{n=1}^\infty$, where
\[
 K_{n,\theta_n}=\theta_nI+(1-\theta_n)K_n.
\]
For $n\ge 1$, let $T_{n,\textnormal{\tiny TV}}^{(c)},T_{n,\textnormal{\tiny TV}}^{(\theta)}$ be the total variation mixing times of the $n$th chains in $\mathcal{F}_c,\mathcal{F}_\theta$. Suppose $\inf_n\theta_n>0$ and there is $\epsilon_0\in(0,1)$ such that $T_{n,\textnormal{\tiny TV}}^{(c)}(\mu_n,\epsilon_0)\ra\infty$ or $T_{n,\textnormal{\tiny TV}}^{(\theta)}(\mu_n,\epsilon_0)\ra\infty$. In the total variation,
\begin{itemize}
\item[(1)] $\mathcal{F}_c$ has a cutoff if and only if $\mathcal{F}_\theta$ has a cutoff. Further, if $t_n$ is a cutoff time for $\mathcal{F}_c$, then $t_n/(1-\theta_n)$ is a cutoff time for $\mathcal{F}_\theta$.

\item[(2)] $\mathcal{F}_c$ has a $(t_n,b_n)$ cutoff if and only if $\mathcal{F}_\theta$ has a $(t_n/(1-\theta_n),b_n)$ cutoff. Further, if $\mathcal{F}_c$ has a $(t_n,b_n)$ cutoff, then $\sqrt{t_n}=O(b_n)$.
\end{itemize}
The above also holds for families without prescribed initial distributions and in the Hellinger distance.
\end{prop}
\begin{proof}
For the total variation, we discuss (2) in detail, while (1) can be shown similarly. In the case that $\theta$ is a constant sequence, Proposition \ref{p-comp2} is exactly the combination of Theorems 3.1, 3.3 and 3.4 in \cite{CSal13-1}. For any sequence $\theta=(\theta_n)_{n=1}^\infty$, we set
\[
 \theta_0:=\inf_{n\ge 1}\theta_n,\quad K_n'=\frac{(\theta_n-\theta_0)I+(1-\theta_n)K_n}{1-\theta_0},\quad H_{n,t}'=e^{-t(I-K_n')t}.
\]
Clearly, one has
\begin{equation}\label{eq-kn'}
 K_{n,\theta_n}=\theta_0I+(1-\theta_0)K_n',\quad H_{n,t}'=H_{n,\frac{1-\theta_n}{1-\theta_0}t}.
\end{equation}
By setting $\zeta=(\zeta_n)_{n=1}^\infty$, where $\zeta_n=\theta_0$, and $\mathcal{F}'=(\mu_n,\mathcal{X}_n,K_n',\pi_n)_{n=1}^\infty$, the first identity in (\ref{eq-kn'}) implies $\mathcal{F}_\theta=\mathcal{F}'_\zeta$, which leads to
\[
 \mathcal{F}_\theta\text{ has a $(r_n,b_n)$ cutoff}\quad\Lra\quad
 \mathcal{F}'_c\text{ has a $((1-\theta_0)r_n,b_n)$ cutoff},
\]
and the second identity yields
\[
 \mathcal{F}_c \text{ has a $(t_n,b_n)$ cutoff} \quad\Lra\quad \mathcal{F}'_c\text{ has a $(\tfrac{1-\theta_0}{1-\theta_n}t_n,b_n)$ cutoff}.
\]
The desired equivalence is then given by the setting of $r_n=t_n/(1-\theta_n)$.

The conclusion for the Hellinger distance follows immediately from Proposition \ref{p-comp0} and what is proved above.
\end{proof}

\section{Distances of product chains}\label{s-dpc}
In this section, we consider product chains and provide bounds on their total variation and Hellinger distance. Let $(\mathcal{X}_i,K_i,\pi_i)_{i=1}^n$ be irreducible Markov chains and $p_1,...,p_n$ be positive reals satisfying $p_1+\cdots+p_n=1$. Referring to the setting in (\ref{eq-prodxpi})-(\ref{eq-prodk}), we call $(\mathcal{X},K,\pi)$ the product chain of $(\mathcal{X}_i,K_i,\pi_i)_{i=1}^n$ according to the probability vector $(p_1,...,p_n)$, call $(\mathcal{X}_i,K_i,\pi_i)$ the $i$th coordinate chain of $(\mathcal{X},K,\pi)$ and name $n$ as its dimension. In the continuous time case, we write $H_{i,t}=e^{-t(I-K_i)}$ and $H_t=e^{-t(I-K)}$. As is stated in the introduction, one has (\ref{eq-prodcts}) but this could fail in the discrete time case.

Throughout this section, we concentrate on the study of continuous time chains. Recall that $d_H^{(c)},d_{i,H}^{(c)}$ and $d_{\textnormal{\tiny TV}}^{(c)},d_{i,\textnormal{\tiny TV}}^{(c)}$ refer to the Hellinger distances and the total variations of $(\mathcal{X},H_t,\pi)$ and $(\mathcal{X}_i,H_{i,t},\pi_i)$ and that $T_H^{(c)},T_{i,H}^{(c)}$ and $T_{\textnormal{\tiny TV}}^{(c)},T_{i,\textnormal{\tiny TV}}^{(c)}$ denote the corresponding mixing times.

\subsection{Distances with prescribed initial distributions}

Our first result is to bound distances of product chains using those of their coordinate chains.

\begin{lem}\label{l-prodmixing}
Let $(\mathcal{X},K,\pi)$ be the product chain of $(\mathcal{X}_i,K_i,\pi_i)_{i=1}^n$ according to the probability vector $(p_1,...,p_n)$. For probability distributions $\mu_1,...,\mu_n$ on $\mathcal{X}_1,...,\mathcal{X}_n$ and the product measure $\mu=\mu_1\times\cdots\times\mu_n$, one has
\begin{equation}\label{eq-prodhd}
 d_H^{(c)}(\mu,t)^2=1-\prod_{i=1}^n\left(1-d_{i,H}^{(c)}(\mu_i,p_it)^2\right)\ge\max_{1\le i\le n}d_{i,H}^{(c)}(\mu_i,p_it)^2.
\end{equation}
and $d_{\textnormal{\tiny TV}}^{(c)}(\mu,t)\ge \max\{d_{\textnormal{\tiny TV}}^{(c)}(\mu_i,p_it):1\le i\le n\}$ and
\[
 1-\prod_{i=1}^n\left(1-d_{i,\textnormal{\tiny TV}}^{(c)}(\mu_i,p_it)^2\right)^{1/2}\le d_{\textnormal{\tiny TV}}^{(c)}(\mu,t)\le 1-\prod_{i=1}^n\left(1-d_{i,\textnormal{\tiny TV}}^{(c)}(\mu_i,p_it)\right).
\]
The above also holds for the maximum total variation and Hellinger distance.
\end{lem}

\begin{proof}
For distances with prescribed initial distributions, the proof is given by Proposition \ref{p-comp} and (\ref{eq-prodcts}) and, for the maximum distances, the proof follows immediately from the fact of $d_{\textnormal{\tiny TV}}^{(c)}(t)=\sup_xd_{\textnormal{\tiny TV}}^{(c)}(\delta_x,t)$, $d_H^{(c)}(t)=\sup_xd_H^{(c)}(\delta_x,t)$ and $\delta_x=\delta_{x_1}\times\cdots\times\delta_{x_n}$ for $x=(x_1,...,x_n)\in\mathcal{X}$.
\end{proof}

The next proposition is an extension of Lemma \ref{l-prodmixing} and could be more applicable to practical computations.

\begin{prop}\label{p-prodmixing}
Let $(\mu_i,\mathcal{X}_i,K_i,\pi_i)_{i=1}^n$ and $(\mu,\mathcal{X},K,\pi)$ be the Markov chains in Lemma \ref{l-prodmixing} and set $\varrho_H=2\varrho_{\textnormal{\tiny TV}}=2$. For $*\in\{H,\textnormal{\tiny TV}\}$, one has
\[
 d_*^{(c)}(\mu,t)^{\varrho_*}\le 1-\exp\left\{-\sum_{i=1}^n\frac{d_{i,*}^{(c)}(\mu_i,p_it)^{\varrho_*}}
 {1-d_{i,*}^{(c)}(\mu_i,p_it)^{\varrho_*}}\right\}
\]
and
\[
 d_*^{(c)}(\mu,t)^{\varrho_*}\ge 1-\exp\left\{-\frac{\varrho_*}{2}\sum_{i=1}^nd_{i,*}^{(c)}(\mu_i,p_it)^2\right\}\wedge\left(1-\max_{1\le i\le n}d_{i,*}^{(c)}(\mu_i,p_it)^{\varrho_*}\right).
\]
In particular, for $A\in (0,1)$,
\begin{equation}\label{eq-prodmixupper}
 d_*^{(c)}(\mu,t)^{\varrho_*}
 \le 1-\exp\left\{-c_A\sum_{i=1}^nd_{i,*}^{(c)}(\mu_i,p_it)^{\varrho_*}\right\},\quad\forall t\ge t_*^{(c)}(A^{1/\varrho_*}),
\end{equation}
where $c_A=1/(1-A)$ and $t_*^{(c)}(A)=\max\{T_{i,*}^{(c)}(\mu_i,A)/p_i|1\le i\le n\}$.

The above also holds for the maximum Hellinger distance and total variation.
\end{prop}
\begin{proof}
In the Hellinger distance, the proofs for the first two inequalities are given by Lemma \ref{l-prodmixing} and the following fact,
\[
 \frac{-u}{1-u}\le -\log\left(1+\frac{u}{1-u}\right)=\log(1-u)\le -u,\quad\forall u\in(0,1),
\]
while the last inequality is implied by the first one with the additional observation $d_{i,H}^{(c)}(\mu_i,p_it)\le \sqrt{A}$ for $t\ge t_H^{(c)}(\sqrt{A})$ and $1\le i\le n$. In the total variation, the proofs are similar and skipped.
\end{proof}

\begin{rem}
By Lemma \ref{l-prodmixing}, one may use the following inequality
\[
 1-(1-a_1)\times\cdots\times(1-a_n)\le a_1+\cdots+a_n,\quad\forall a_1,...,a_n\in[0,1]
\]
to obtain
\begin{equation}\label{eq-distsum}
 d_H^{(c)}(\mu,t)^2\le\sum_{i=1}^nd_{i,H}^{(c)}(\mu_i,t)^2,\quad d_{\textnormal{\tiny TV}}^{(c)}(\mu,t)\le\sum_{i=1}^nd_{i,\textnormal{\tiny TV}}^{(c)}(\mu_i,t).
\end{equation}
Compared with the last inequality in Proposition \ref{p-prodmixing}, (\ref{eq-distsum}) provides simpler upper bounds without the requirement of $t\ge t_H^{(c)}(A)$ and $t\ge t_{\textnormal{\tiny TV}}^{(c)}(A)$.
\end{rem}

The following example is an illustration of Lemma \ref{l-prodmixing} and Proposition \ref{p-prodmixing}.

\begin{ex}\label{ex-2p}
Let $n\in\mathbb{N}$ and $\alpha,\beta\in(0,1]$. For $1\le i\le n$, let $(\mathcal{X}_i,K_i,\pi_i)$ be a Markov chain with
\begin{equation}\label{eq-twopoint}
 \mathcal{X}_i=\{0,1\},\quad K_i=\left(\begin{array}{cc}1-\alpha&\alpha\\\beta&1-\beta\end{array}\right),\quad
 \pi_i=\left(\frac{\beta}{\alpha+\beta},\frac{\alpha}{\alpha+\beta}\right)
\end{equation}
and $(\mathcal{X}_i,H_{i,t},\pi_i)$ be the continuous time Markov chain associated with $(\mathcal{X}_i,K_i,\pi_i)$. By (\ref{eq-2pK}), one has
\[
 H_{i,t}(0,0)=\frac{\beta}{\alpha+\beta}+\frac{\alpha}{\alpha+\beta}e^{-(\alpha+\beta)t}.
\]
This implies
\begin{equation}\label{eq-2phd}
 d_{i,H}^{(c)}(0,t)^2=1-\frac{\beta}{\alpha+\beta}\sqrt{1+\frac{\alpha}{\beta}e^{-(\alpha+\beta)t}}-\frac{\alpha}{\alpha+\beta}
 \sqrt{1-e^{-(\alpha+\beta)t}}
\end{equation}
and
\begin{equation}\label{eq-2ptv}
 d_{i,\text{\tiny TV}}^{(c)}(0,t)=\frac{\alpha}{\alpha+\beta}e^{-(\alpha+\beta)t}.
\end{equation}

Let $K$ be the product chain of $(\mathcal{X}_i,K_i,\pi_i)_{i=1}^n$ according to the probability vector $(1/n,...,1/n)$ and consider the case $\alpha=\beta$. Note that, from Lemma \ref{l-comp2} and (\ref{eq-2ptv}),
\[
 d_{i,H}^{(c)}(0,t)^2\le d_{i,\text{\tiny TV}}^{(c)}(0,t)=\frac{1}{2}e^{-2\alpha t}\le\frac{1}{2},\quad\forall t\ge 0.
\]
By applying (\ref{eq-prodmixupper}) with $A=1/2$, one has
\[
 1-\exp\left\{-\frac{n(ne^c)^{-4a\alpha}}{8}\right\}\le d_{\text{\tiny TV}}^{(c)}(\mathbf{0},an(\log n+c))\le 1-\exp\left\{-n(ne^c)^{-2a\alpha}\right\},
\]
and
\[
 1-\exp\left\{-f_n(c)\right\}\le d_H^{(c)}\left(\mathbf{0},(4\alpha)^{-1}n(\log n+c)\right)^2\le1-\exp\left\{-2f_n(c)\right\},
\]
for $a>0$ and $c>-\log n$, where $\mathbf{0}=(0,...,0)$ and
\begin{align}
 f_n(c)&=\frac{n}{2}\left(2-\sqrt{1+(ne^c)^{-1/2}}-\sqrt{1-(ne^c)^{-1/2}}\right)\notag\\
 &=\frac{e^{-c}}{\left(2+\sqrt{2+2\sqrt{1-(ne^c)^{-1}}}\right)
 \left(1+\sqrt{1-(ne^c)^{-1}}\right)}.\notag
\end{align}
The last equality yields $e^{-c}/8\le f_n(c)\le e^{-c}/(2\sqrt{2})$ and the bounds for distances lead to
\[
 \frac{n}{4\alpha}\left(\log n-\log\log\frac{1}{1-\epsilon}-3\right)\le T_{\text{\tiny TV}}^{(c)}(\mathbf{0},\epsilon)\le\frac{n}{2\alpha}\left(\log n-\log\log\frac{1}{1-\epsilon}\right)
\]
and
\[
 \frac{n}{4\alpha}\left(\log n-\log\log\frac{1}{1-\epsilon}-3\right)\le T_H^{(c)}(\mathbf{0},\epsilon)\le \frac{n}{4\alpha}\left(\log n-\log\log\frac{1}{1-\epsilon}\right),
\]
for $0<\epsilon<1-e^{-n}$. Consequently, we may conclude that, when $n$ tends to infinity, the total variation mixing time has order $n\log n$, while the mixing time in the Hellinger distance is asymptotically $(4\alpha)^{-1}n\log n$.

Next, we make a more precise estimation of the total variation using Lemma \ref{l-comp2} and Lemma \ref{l-prodmixing}. By (\ref{eq-prodhd}) and (\ref{eq-2phd}), one has
\begin{align}
 d_H^{(c)}\left(\mathbf{0},(4\alpha)^{-1} n(\log n+c)\right)^2&=1-\left(\frac{\sqrt{1+(e^cn)^{-1/2}}
 +\sqrt{1-(e^cn)^{-1/2}}}{2}\right)^n\notag\\
 &=1-\exp\left\{-\frac{1}{8e^c}\right\}+O\left(\frac{1}{n}\right),\notag
\end{align}
for $c\in\mathbb{R}$, where the last equality is the result of the fact that, as $t\ra 0$,
\[
 \frac{\sqrt{1+t}+\sqrt{1-t}}{2}=1-\frac{t^2}{8}+O(t^4),\,\, \log(1-t)=1-t-O(t^2),\,\, e^{-t}=1-O(t).
\]
By (\ref{eq-hdtv0}), this implies
\[
 d_{\text{\tiny TV}}^{(c)}\left(\mathbf{0},(4\alpha)^{-1} n(\log n+c)\right)\le\sqrt{1-\exp\left\{-\frac{1}{4e^c}\right\}}+O\left(\frac{1}{n}\right)
\]
and
\[
 d_{\text{\tiny TV}}^{(c)}\left(\mathbf{0},(4\alpha)^{-1} n(\log n+c)\right)\ge
 1-\exp\left\{-\frac{1}{8e^c}\right\}+O\left(\frac{1}{n}\right).
\]
Consequently, we may conclude that the total variation mixing time is also asymptotically $(4\alpha)^{-1}n\log n$.
\end{ex}

\subsection{Maximum distances of product chains}

In this subsection, we consider distances of product chains in the sense of (\ref{eq-hdtvdf}) and our first result is the application of Proposition \ref{p-prodmixing} to the total variation.

\begin{prop}\label{p-prodmaxtv}
Let $(\mathcal{X},K,\pi)$ be the product chain of $(\mathcal{X}_i,K_i,\pi_i)_{i=1}^n$ according to the probability vector $(p_1,...,p_n)$. For $\epsilon_i\in(0,1/2)$ and $u_i\ge T_{i,\textnormal{\tiny TV}}^{(c)}(\epsilon_i)$ with $1\le i\le n$, one has
\[
 d_{\textnormal{\tiny TV}}^{(c)}(t)\le 1-\exp\left\{-\sum_{i=1}^n(2\epsilon_i)^{\lfloor p_it/u_i\rfloor}\right\},\quad\forall t\ge \max_{1\le i\le n}\frac{u_i}{p_i}.
\]
\end{prop}

\begin{proof}
By (\ref{eq-prodmixupper}), one has
\[
 d^{(c)}_{\text{\tiny TV}}(t)\le 1-\exp\left\{-2\sum_{i=1}^nd^{(c)}_{i,\text{\tiny TV}}(p_it)\right\},\quad\forall t\ge\max_{1\le i\le n}\frac{T^{(c)}_{i,\text{\tiny TV}}(1/2)}{p_i}.
\]
Since $s\mapsto 2d^{(c)}_{i,\text{\tiny TV}}(s)$ is submultiplicative, we have
\[
 2d^{(c)}_{i,\text{\tiny TV}}(p_it)\le 2d^{(c)}_{i,\text{\tiny TV}}\left(u_i\left\lfloor\frac{p_it}{u_i}\right\rfloor\right)\le \left(2d^{(c)}_{i,\text{\tiny TV}}(u_i)\right)^{\lfloor p_it/u_i\rfloor}\le
 (2\epsilon_i)^{\lfloor p_it/u_i\rfloor}.
\]
As $u_i\ge T^{(c)}_{i,\text{\tiny TV}}(1/2)$ for all $1\le i\le n$, the above inequalities combine to the desired one.
\end{proof}

To get a variant of Proposition \ref{p-prodmaxtv} in the maximum Hellinger distance, one may follow the same reasoning as before but needs the quasi-submultiplicativity of $d_H^{(c)}(t)$, which is the submultiplicativity of $Cd_H^{(c)}(t)$ for some $C>0$. To see a lower bound on $C$, let's consider the two-state chain in (\ref{eq-2sc}). By (\ref{eq-2phd}), when $\alpha\ge\beta$, one has
\[
 d^{(c)}_H(t)=d^{(c)}_H(0,t)\sim \sqrt{\frac{\alpha}{8\beta}}e^{-2(\alpha+\beta)t},\quad\text{as }t\ra\infty.
\]
Note that if $t\mapsto Cd^{(c)}_H(t)$ is submultiplicative, then $d_H^{(c)}(2t)/(d_H^{(c)}(t))^2\le C$ for all $t\ge 0$. This implies
\[
 C\ge \lim_{t\ra\infty}\frac{d_H^{(c)}(2t)}{(d_H^{(c)}(t))^2}=\sqrt{\frac{8\beta}{\alpha}},\quad\forall \alpha\ge \beta.
\]
Taking $\alpha=\beta$ yields $C\ge \sqrt{8}$.

\begin{lem} {\rm(\cite[Lemma A.3]{CK17})}\label{l-hdsub}
Let $d_H,d_H^{(c)}$ be the maximum Hellinger distances of a discrete time irreducible Markov chain and its associated continuous time one. Then, the following mappings
\[
 m\mapsto 4d_H(m),\quad t\mapsto 4d_H^{(c)}(t),
\]
are non-increasing and submultiplicative.
\end{lem}

The next proposition follows immediately from (\ref{eq-prodmixupper}) and Lemma \ref{l-hdsub}, of which proof is similar to that of Proposition \ref{p-prodmaxtv}.

\begin{prop}\label{p-prodmaxhd}
Referring to the product chain in Proposition \ref{p-prodmaxtv}, one has
\[
 d_H^{(c)}(t)^2\le 1-\exp\left\{-\frac{1}{8}\sum_{i=1}^n(4\epsilon_i)^{2\lfloor p_it/u_i\rfloor}\right\},\quad \forall t\ge \max_{1\le i\le n}\frac{u_i}{p_i},
\]
where $\epsilon_i\in(0,1/\sqrt{2})$ and $u_i\ge T_{i,H}^{(c)}(\epsilon_i)$ for $1\le i\le n$.
\end{prop}

\section{Cutoffs for product chains}\label{s-cpc}

In this section, we consider families of product chains and discuss their cutoffs. Let $(k_n)_{n=1}^\infty$ be a sequence of positive integers and
\begin{equation}\label{eq-fp}
 \mathcal{F}=\{(\mathcal{X}_{n,i},K_{n,i},\pi_{n,i})|1\le i\le k_n,n\ge 1\},\quad
 \mathcal{P}=\{p_{n,i}|1\le i\le k_n,n\ge 1\},
\end{equation}
be a family of irreducible Markov chains and a triangular array of positive reals. For $n\ge 1$, let $(\mathcal{X}_n,K_n,\pi_n)$ be the product chain of $(\mathcal{X}_{n,i},K_{n,i},\pi_{n,i})_{i=1}^{k_n}$ according to the probability vector $(p_{n,1}/q_n,...,p_{n,k_n}/q_n)$, where $q_n=p_{n,1}+\cdots+p_{n,k_n}$. We write $\mathcal{F}^\mathcal{P}$ for the family $(\mathcal{X}_n,K_n,\pi_n)_{n=1}^\infty$ and call it the family of product chains of $\mathcal{F}$ according to $\mathcal{P}$. In the continuous time case, we set $H_{n,i,t}=e^{-t(I-K_{n,i})}$, $H_{n,t}=e^{-t(I-K_n)}$ and $\mathcal{F}^\mathcal{P}_c=(\mathcal{X}_n,H_{n,t},\pi_n)_{n=1}^\infty$.
For the Markov chains, $(\mathcal{X}_{n,i},H_{n,i,t},\pi_{n,i})$ and $(\mathcal{X}_n,H_{n,t},\pi_n)$, we use $d_{n,i,H}^{(c)},d_{n,H}^{(c)}$ and $d_{n,i,\text{\tiny TV}}^{(c)},d_{n,\text{\tiny TV}}^{(c)}$ to denote their Hellinger distances and total variations and write $T_{n,i,H}^{(c)},T_{n,H}^{(c)}$ and $T_{n,i,\text{\tiny TV}}^{(c)},T_{n,\text{\tiny TV}}^{(c)}$ for their corresponding mixing times.

\subsection{Cutoff of product chains in the Hellinger distance}

The following theorem provides equivalent conditions for cutoffs in the Hellinger distance.

\begin{thm}\label{t-prodcutoffhd}
Let $\mathcal{F},\mathcal{P}$ be families in \textnormal{(\ref{eq-fp})}. For $n\ge 1$ and $1\le i\le k_n$, let $\mu_{n,i}$ be a probability on $\mathcal{X}_{n,i}$, $\mu_n=\mu_{n,1}\times\cdots\times\mu_{n,k_n}$, $q_n=p_{n,1}+\cdots+p_{n,k_n}$ and set
\[
 F_n(t)=\frac{\sum_{i=1}^{k_n}d_{n,i,H}^{(c)}\left(\mu_{n,i},p_{n,i}t/q_n
 \right)^2}{1-\max\limits_{1\le i\le k_n}d_{n,i,H}^{(c)}\left(\mu_{n,i},p_{n,i}t/q_n\right)^2},\quad
 G_n(t)=\max_{1\le i\le k_n}d_{n,i,H}^{(c)}\left(\mu_{n,i},p_{n,i}t/q_n
 \right),
\]
where $1/0:=\infty$. Then, the family $\mathcal{F}^\mathcal{P}_c$ with initial distributions $(\mu_n)_{n=1}^\infty$ has:
\begin{itemize}
\item[(1)] a cutoff in the Hellinger distance with cutoff time $t_n$ if and only if
\[
 \lim_{n\ra\infty}F_n(at_n)
 =\begin{cases}0&\text{for }a>1,\\\infty&\text{for }0<a<1,\end{cases}
\]
\item[(2)] a $(t_n,b_n)$ cutoff in the Hellinger distance if and only if $t_n>0$, $b_n>0$, $b_n=o(t_n)$ and
\[
 \lim_{c\ra\infty}\limsup_{n\ra\infty}F_n(t_n+cb_n)=0,\quad
 \lim_{c\ra-\infty}\liminf_{n\ra\infty}F_n(t_n+cb_n)=\infty.
\]
\end{itemize}

In particular, when $\sup_nk_n<\infty$, the equivalences in (1) and (2) remain true under the replacement of $F_n$ and $\infty$ with $G_n$ and $1$.
\end{thm}

\begin{rem}\label{r-prodcutofftv}
In Theorem \ref{t-prodcutoffhd}, when $\sup_nk_n<\infty$, the corresponding conclusion also holds in the total variation.
\end{rem}

\begin{rem}\label{r-prodcutoffmax}
Both Theorem \ref{t-prodcutoffhd} and Remark \ref{r-prodcutofftv} are valid in the maximum distance.
\end{rem}

\begin{proof}[Proof of Theorem \ref{t-prodcutoffhd}]
We deal with the general setting here, while the case of bounded dimensions can be treated similarly. Set
\[
 f_n(t)=\sum_{i=1}^{k_n}\frac{d_{n,i,H}^{(c)}\left(\mu_{n,i},p_{n,i}t/q_n
 \right)^2}{1-\max\limits_{1\le i\le k_n}d_{n,i,H}^{(c)}\left(\mu_{n,i},p_{n,i}t/q_n\right)^2},\quad
 g_n(t)=\sum_{i=1}^{k_n}d_{n,i,H}^{(c)}(\mu_{n,i},p_{n,i}t/q_n)^2.
\]
By the definitions of $F_n,G_n,f_n,g_n$ and Proposition \ref{p-prodmixing}, one has
\[
 G_n^2\le g_n\le \log\frac{1}{1-d_{n,H}^{(c)}(\mu_n,\cdot)^2},\quad f_n\le F_n=\frac{g_n}{1-G_n^2},\quad d_{n,H}^{(c)}(\mu_n,\cdot)^2\le 1-e^{-f_n},
\]
and
\[
 1-d_{n,H}^{(c)}(\mu_n,\cdot)^2\le e^{-g_n}\wedge(1-G_n^2)\le \frac{1}{g_n}\wedge(1-G_n^2)
 \le\sqrt{1/F_n}.
\]
This implies that, for any sequence of positive reals $(t_n)_{n=1}^\infty$,
\[
 \lim_{n\ra\infty}d_{n,H}^{(c)}(\mu_n,t_n)=0\quad\Lra\quad \lim_{n\ra\infty}f_n(t_n)=0 \quad\Lra\quad \lim_{n\ra\infty}F_n(t_n)=0
\]
and
\[
 \lim_{n\ra\infty}d_{n,H}^{(c)}(\mu_n,t_n)=1 \quad\Lra\quad \lim_{n\ra\infty}f_n(t_n)=\infty\quad\Lra\quad\lim_{n\ra\infty}F_n(t_n)=\infty,
\]
which leads to (1). As (2) can be derived in a similar way, we skip the details.
\end{proof}

The next corollary provides a sufficient condition for families of product chain without cutoffs in the Hellinger distance, of which proof is obvious from Proposition \ref{p-prodmixing} and skipped.

\begin{cor}
Refer to Theorem \ref{t-prodcutoffhd}. If there are $t_n>0$ and $b>a>0$ such that
\[
 \limsup_{n\ra\infty}F_n(at_n)<\infty,\quad \liminf_{n\ra\infty}F_n(bt_n)>0,
\]
then no subfamily of $\mathcal{F}^\mathcal{P}_c$ presents a cutoff in the Hellinger distance. The above also holds in the maximum Hellinger distance.
\end{cor}

\subsection{Cutoffs for product chains in the maximum distances}

In this subsection, we discuss the cutoff in the maximum distance. As cutoffs for product chains with bounded dimensions have been highlighted in Remark \ref{r-prodcutoffmax}, we will focus on the case with dimensions tending to infinity thereafter.

\begin{thm}\label{t-maxcuttv}
Let $\mathcal{F},\mathcal{P}$ be families in \textnormal{(\ref{eq-fp})} with $k_n\ra\infty$ and $\{t_{n,i}>0:1\le i\le k_n,\,n\ge 1\}$ be a family of positive reals. Set $\varrho_H=2\varrho_{\textnormal{\tiny TV}}=2$ and, for $*\in\{\textnormal{\tiny TV},H\}$,
\[
 s_n=\max_{1\le i\le k_n}\frac{T_{n,i,*}^{(c)}(\epsilon_{n,i})}{p_{n,i}},\quad R(c)=\lim_{m\ra\infty}\limsup_{n\ra\infty}\sum_{i=1}^{k_n-m}(2\varrho_*\epsilon_{n,i})^{c\varrho_*s_np_{n,i}/
 T_{n,i,*}^{(c)}(\epsilon_{n,i})}.
\]
Suppose $\epsilon_{n,i}\in(0,1/(2\varrho_*))$ and $\inf_{i,n}\epsilon_{n,i}>0$. Then, for $*\in\{\textnormal{\tiny TV},H\}$,
\begin{itemize}
\item[(1)] If $R(c)=0$ for all $c>1$ and the family $(\mathcal{X}_{n,k_n-m},H_{n,k_n-m,t},\pi_{n,k_n-m})_{n=1}^\infty$ has a cutoff, for all $m\ge 0$, with cutoff time $(t_{n,k_n-m})_{n=1}^\infty$, then $\mathcal{F}^\mathcal{P}_c$ has a cutoff with cutoff time $(p_{n,1}+\cdots+p_{n,k_n})s_n$ and $\limsup_ns_n/t_n\le 1$, where
    $t_n:=\max\{t_{n,i}/p_{n,i}:1\le i\le k_n\}$. If $\sup_{i,n}\{t_{n,i}/T_{n,i,*}^{(c)}(\epsilon_{n,i})\}<\infty$ is assumed further, then $s_n\sim t_n$.

\item[(2)] If $R(c)=0$ for some $c\in(0,1)$ and $\mathcal{F}^\mathcal{P}_c$ has a cutoff with cutoff time $v_n$, then there are sequences of positive integers $(j_n)_{n=1}^\infty$ and $(J_n)_{n=1}^\infty$ satisfying $j_n>j_{n-1}$, $1\le J_n\le k_{j_n}$ and $|k_{j_n}-J_n|=O(1)$ such that the family $(\mathcal{X}_{j_n,J_n},H_{j_n,J_n,t},\pi_{j_n,J_n})_{n=1}^\infty$ has a cutoff with cutoff time $p_{j_n,J_n}v_{j_n}/q_{j_n}$.
\end{itemize}
\end{thm}

The proof of Theorem \ref{t-maxcuttv} is tricky and we discuss it in the next subsection. In Theorem \ref{t-maxcuttv}, it is easy to check that $R(c)$ is non-increasing in $c$. Note that $\epsilon_{n,i}<1/(2\varrho_*)$ is sufficient for $T_{n,i,*}^{(c)}(\epsilon_{n,i})>0$ and that $\epsilon_{n,k_n-m}<1/(2\varrho_*)$ for $n,m$ large enough is necessary for $R(c)=0$. When $R(c)=0$, one can see from Proposition \ref{p-prodmaxtv} that, for the total variation or the Hellinger distance of the $n$th chain at time $cs_n$, the contribution from all but the last finitely many chains in $(\mathcal{X}_{n,i},K_{n,i},\pi_{n,i})_{i=1}^{k_n}$ is asymptotically negligible. In the following, we introduce more properties of $R(c)$ which are useful in proving and applying Theorem \ref{t-maxcuttv}.

\begin{lem}\label{l-prod}
Refer to Theorem \ref{t-maxcuttv} and assume $\epsilon_{n,i}\in(0,1/(2\varrho_*))$ and $\inf_{i,n}\epsilon_{n,i}>0$. For $*\in\{\textnormal{\tiny TV},H\}$, one has:
\begin{itemize}
\item[(1)] If $R(c)=0$ for some $c>0$, then, for any $\delta\in(0,1)$, there are positive integers $N>M>0$ such that
\[
 \max_{1\le i\le k_n-M}\frac{T_{n,i,*}^{(c)}(\epsilon_{n,i})}{p_{n,i}}
 \le \delta\max_{k_n-M<i\le k_n}\frac{T_{n,i,*}^{(c)}(\epsilon_{n,i})}{p_{n,i}},\quad\forall n\ge N.
\]

\item[(2)] If $R(c_0)<\infty$ for some $c_0>0$, then either $R(c)>0$ for $c>c_0$ or $R(c)=0$ for $c>c_0$.

\item[(3)] $R(c)=0$ for some $c\in(0,1)$ if and only if $R(c)=0$ for some $c>0$ and $R(c)<\infty$ for some $c\in(0,1)$.

\item[(4)] $R(c)<\infty$ if and only if $\sup_{n\ge 1}
\sum_{i=1}^{k_n}(2\varrho_*\epsilon_{n,i})^{c\varrho_*s_np_{n,i}/T^{(c)}_{n,i,*}
(\epsilon_{n,i})}<\infty$.
\end{itemize}
\end{lem}
\begin{proof}[Proof of Lemma \ref{l-prod}]
Note that (3) is a corollary of (2), while (2) is an immediate result of Lemma \ref{l-analytic2}. (4) is clear from the definition of $R(c)$. To see (1), we set $\alpha=\inf_{i,n}\epsilon_{n,i}$. By the following inequality
\[
 \sum_{i=1}^{k_n-m}(2\varrho_*\epsilon_{n,i})^{c\varrho_*s_np_{n,i}/T_{n,i,*}^{(c)}(\epsilon_{n,i})}\ge (2\varrho_*\alpha)^{c\varrho_*s_n\min\{p_{n,i}/T_{n,i,*}^{(c)}(\epsilon_{n,i}):1\le i\le k_n-m\}},
\]
if $R(c)=0$ for some $c>0$, then
\[
 \lim_{m\ra\infty}\liminf_{n\ra\infty}\frac{\max\{T_{n,i,*}^{(c)}(\epsilon_{n,i})/p_{n,i}:1\le i\le k_n\}}{\max\{T_{n,i,*}^{(c)}(\epsilon_{n,i})/p_{n,i}:1\le i\le k_n-m\}}=\infty,
\]
which is equivalent to the conclusion in (1).
\end{proof}

\begin{proof}[Proof of Theorem \ref{t-maxcuttv}]
We will prove Theorem \ref{t-maxcuttv} in the total variation, while the variant in the Hellinger distance can be treated in a similar way and skipped. First, we set up some notations and make basic analysis. For convenience, set $\alpha=\inf_{i,n}\epsilon_{n,i}$, $q_n=p_{n,1}+\cdots+p_{n,k_n}$ and $u_{n,i}=T_{n,i,\textnormal{\tiny TV}}^{(c)}(\epsilon_{n,i})$. Clearly, $s_n=\max\{u_{n,i}/p_{n,i}:1\le i\le k_n\}$. Let $d^{(c)}_{n,i,\text{\tiny TV}}$ and $d^{(c)}_{n,\text{\tiny TV}}$ be the total variations of $(\mathcal{X}_{n,i},H_{n,i},\pi_{n,i})$ and $(\mathcal{X}_n,H_{n,t},\pi_n)$. By the second inequality of Proposition \ref{p-prodmixing}, one has
\begin{equation}\label{eq-lb}
 d^{(c)}_{n,\text{\tiny TV}}(t)\ge \max_{1\le i\le k_n}d^{(c)}_{n,i,\text{\tiny TV}}(p_{n,i}t/q_n),\quad\forall t>0.
\end{equation}
Let $(\mathcal{X}_n^{\mathcal{L}},K_n^{\mathcal{L}},\pi_n^{\mathcal{L}})$ and $(\mathcal{X}_n^{\mathcal{R}},K_n^{\mathcal{R}},\pi_n^{\mathcal{R}})$ be the product chains of
\[
 (\mathcal{X}_{n,i},K_{n,i},\pi_{n,i})_{i=1}^{k_n-m},\quad (\mathcal{X}_{n,i},K_{n,i},\pi_{n,i})_{i=k_n-m+1}^{k_n},
\]
according to the probability vectors
\[
 (p_{n,1}/q_n^\mathcal{L},...,p_{n,k_n-m}/q_n^\mathcal{L}),\quad (p_{n,k_n-m+1}/q_n^\mathcal{R},...,p_{n,k_n}/q_n^\mathcal{R}),
\]
where $q_n^\mathcal{L}=\sum_{i=1}^{k_n-m}q_{n,i}$ and $q_n^\mathcal{R}=\sum_{i=k_n-m+1}^{k_n}q_{n,i}$. Obviously, $(\mathcal{X}_n,K_n,\pi_n)$ is the product chain of $(\mathcal{X}_n^{\mathcal{L}},K_n^{\mathcal{L}},\pi_n^{\mathcal{L}})$ and $(\mathcal{X}_n^{\mathcal{R}},K_n^{\mathcal{R}},\pi_n^{\mathcal{R}})$ according to the probability vector $(q_n^\mathcal{L}/q_n,q_n^\mathcal{R}/q_n)$. Let $d_{n,\text{\tiny TV}}^{\mathcal{L},{(c)}}$ and $d_{n,\text{\tiny TV}}^{\mathcal{R},{(c)}}$ be the maximum total variations of the continuous time chains associated with $(\mathcal{X}_n^{\mathcal{L}},K_n^{\mathcal{L}},\pi_n^{\mathcal{L}})$ and
$(\mathcal{X}_n^{\mathcal{R}},K_n^{\mathcal{R}},\pi_n^{\mathcal{R}})$. By Lemma \ref{l-prodmixing}, one has
\[
 d_{n,\text{\tiny TV}}^{(c)}(t)\le 1-\left(1-d_{n,\text{\tiny TV}}^{\mathcal{L},{(c)}}(q_n^\mathcal{L}t/q_n)\right)\left(1-d_{n,\text{\tiny TV}}^{\mathcal{R},{(c)}}(q_n^\mathcal{R}t/q_n)\right),
\]
and
\[
 d_{n,\text{\tiny TV}}^{\mathcal{R},{(c)}}(q_n^\mathcal{R}t/q_n)
 \le 1-\prod_{i=k_n-m+1}^{k_n}\left(1-d_{n,i,\text{\tiny TV}}^{(c)}(p_{n,i}t/q_n)\right).
\]
Further, by Proposition \ref{p-prodmaxtv}, when $t\ge q_n\max\{u_{n,i}/p_{n,i}:1\le i\le k_n-m\}$, one has
\begin{align}
 d_{n,\text{\tiny TV}}^{\mathcal{L},{(c)}}(q_n^\mathcal{L}t/q_n)&\le 1-\exp\left\{-\sum_{i=1}^{k_n-m}(2\epsilon_{n,i})^{\lfloor (p_{n,i}t)/(u_{n,i}q_n)\rfloor}\right\}\notag\\
 &\le 1-\exp\left\{-\frac{1}{2\alpha}\sum_{i=1}^{k_n-m}
 (2\epsilon_{n,i})^{(p_{n,i}t)/(u_{n,i}q_n)}\right\}.\notag
\end{align}
As a consequence of the above inequalities, we have
\begin{equation}\label{eq-ub}
\begin{aligned}
 d^{(c)}_{n,\text{\tiny TV}}(t)\le &1-\exp\left\{-\frac{1}{2\alpha}\sum_{i=1}^{k_n-m}
 (2\epsilon_{n,i})^{(p_{n,i}t)/(u_{n,i}q_n)}\right\}\\
 &\hskip0.5in\times\prod_{i=k_n-m+1}^{k_n}\left(1-d^{(c)}_{n,i,\text{\tiny TV}}(p_{n,i}t/q_n)\right),
\end{aligned}
\end{equation}
for $t\ge q_n\max\{u_{n,i}/p_{n,i}:1\le i\le k_n-m\}$.

To prove (1), we assume $R(c)=0$ for $c>1$ and
\begin{equation}\label{eq-cutarray}
 \lim_{n\ra\infty}d_{n,k_n-m,\text{\tiny TV}}^{(c)}(ct_{n,k_n-m})=\begin{cases}0&\text{for }c>1,\\1&\text{for }0<c<1,\end{cases}
\end{equation}
for any $m\ge 0$. For $n\ge 1$, let $1\le \ell_n\le k_n$ be a positive integer such that $s_n=u_{n,\ell_n}/p_{n,\ell_n}$. By Lemma \ref{l-prod}(1), there are positive integers $N>M$ such that $k_n-M\le \ell_n\le k_n$ for $n\ge N$. By Lemma \ref{l-cutoff}, since $\inf_{i,n}\epsilon_{n,i}>0$ and $\sup_{i,n}\epsilon_{n,i}<1$, (\ref{eq-cutarray}) also holds as $t_{n,k_n-m}$ is replaced by $u_{n,k_n-m}$. Consequently, (\ref{eq-cutarray}) remains true when $k_n-m$ and $t_{n,k_n-m}$ are replaced by $\ell_n$ and $u_{n,\ell_n}$. By (\ref{eq-lb})-(\ref{eq-ub}), we obtain
\[
  \forall c\in(0,1),\quad\liminf_{n\ra\infty}d^{(c)}_{n,\text{\tiny TV}}(cq_ns_n)\ge\liminf_{n\ra\infty}d^{(c)}_{n,\ell_n,\text{\tiny TV}}(cu_{n,\ell_n})=1,
\]
and
\[
 \forall c>1,\quad\limsup_{n\ra\infty}d^{(c)}_{n,\text{\tiny TV}}(cq_ns_n)
 \le 1-\exp\left\{-R(c)/2\alpha\right\}=0.
\]
This proves that $\mathcal{F}_c^\mathcal{P}$ has a total variation cutoff with cutoff time $q_ns_n$.

Next, we compare $s_n$ and $t_n$, where $t_n:=\max\{t_{n,i}/p_{n,i}:1\le i\le k_n\}$. By Lemma \ref{l-prod}(1), one may choose, for any $\delta\in(0,1)$, two integers $N_\delta>M_\delta>0$ such that
\begin{equation}\label{eq-MdeltaNdelta}
 \max\left\{\frac{u_{n,i}}{p_{n,i}}:1\le i\le k_n-M_\delta\right\}\le\delta s_n,\quad\forall n\ge N_\delta.
\end{equation}
This implies
\begin{equation}\label{eq-sn<tn}
 s_n=\max\left\{\frac{u_{n,i}}{p_{n,i}}:k_n-M_\delta<i\le k_n\right\}\le A_{n,\delta}t_n,\quad\forall n\ge N_\delta,
\end{equation}
where $A_{n,\delta}=\max\left\{u_{n,i}/t_{n,i}:k_n-M_\delta<i\le k_n\right\}$, and
\begin{equation}\label{eq-tn<sn}
\begin{aligned}
 t_n&=\max\left\{\frac{t_{n,i}}{p_{n,i}}:1\le i\le k_n-M_\delta\right\}
 \vee\max\left\{\frac{t_{n,i}}{p_{n,i}}:k_n-M_\delta<i\le k_n\right\}\\ &\le(C\delta s_n)\vee(B_{n,\delta}s_n),\quad\forall n\ge N_\delta,
 \end{aligned}
\end{equation}
where $B_{n,\delta}=\max\left\{t_{n,i}/u_{n,i}:k_n-M_\delta<i\le k_n\right\}$ and $C=\sup_{i,n}t_{n,i}/u_{n,i}$. Since $(\mathcal{X}_{n,k_n-m},H_{n,k_n-m,t},\pi_{n,k_n-m})_{n=1}^\infty$ has a cutoff, one has $\lim_nt_{n,k_n-m}/u_{n,k_n-m}=1$ for any $m\ge 0$, which leads to $\lim_nA_{n,\delta}=\lim_nB_{n,\delta}=1$ for all $\delta\in(0,1)$. Immediately, (\ref{eq-sn<tn}) implies $\limsup_ns_n/t_n\le 1$. Moreover, if $C<\infty$, then applying (\ref{eq-tn<sn}) with $\delta=(C+1)^{-1}$ yields $\limsup_nt_n/s_n\le 1$.

To prove (2), we assume that $R(c_0)=0$ for some $c_0\in(0,1)$ and $\mathcal{F}_c^\mathcal{P}$ has a cutoff with cutoff time $v_n$. By (\ref{eq-lb}), one has
\begin{equation}\label{eq-cpv/q}
 \lim_{n\ra\infty}\max_{1\le i\le k_n}d^{(c)}_{n,i,\text{\tiny TV}}(cp_{n,i}v_n/q_n)=0,\quad\forall c>1.
\end{equation}
Since $\epsilon_{n,i}\ge \alpha>0$, (\ref{eq-cpv/q}) implies that, for any $c>1$, there is $n_c>0$ such that $u_{n,i}\le cp_{n,i}v_n/q_n$ for all $1\le i\le k_n$ and $n\ge n_c$. Clearly, this is equivalent to
\begin{equation}\label{eq-snvn}
 \limsup_{n\ra\infty}\frac{s_n}{v_n/q_n}\le 1.
\end{equation}
Next, we set
\[
 \beta(c)=\sup_{n\ge 1}\sum_{i=1}^{k_n}(2\epsilon_{n,i})^{(cp_{n,i}v_n)/(u_{n,i}q_n)}, \quad\forall c>0.
\]
By Lemma \ref{l-prod}(4) and the fact of $R(c_0)<\infty$, one has
\begin{equation}\label{eq-betac}
 \beta(c)<\infty,\quad \forall c>c_0.
\end{equation}
Let $N_\delta>M_\delta$ be the constants such that (\ref{eq-MdeltaNdelta}) holds, set $M'=M_{c_0/2}$ and, by (\ref{eq-snvn}), select $N'\ge N_{c_0/2}$ such that $s_n\le 2v_n/q_n$ for $n\ge N'$. As a result of (\ref{eq-MdeltaNdelta}), this implies
\begin{equation}\label{eq-N'}
 \max\left\{\frac{u_{n,i}}{p_{n,i}}:1\le i\le k_n-M'\right\}\le \frac{c_0v_n}{q_n},\quad \forall n>N'.
\end{equation}
Immediately, one may use (\ref{eq-ub}) and (\ref{eq-N'}) to obtain
\begin{equation}\label{eq-beta}
 d^{(c)}_{n,\text{\tiny TV}}(cv_n)\le 1-e^{-\beta(c)/(2\alpha)}
 \left(1-\max_{k_n-M'<i\le k_n}d^{(c)}_{n,i,\text{\tiny TV}}(cp_{n,i}v_n/q_n)\right)^{M'},
\end{equation}
for $c>c_0$ and $n>N'$.

Now, let $c_r\in(c_0,1)$ be an increasing sequence converging to $1$. By (\ref{eq-betac}), we have $\beta(c_r)<\infty$. As $\mathcal{F}_c^\mathcal{P}$ has a cutoff with cutoff time $v_n$, one may use (\ref{eq-beta}) to derive
\[
 \lim_{n\ra\infty}\max_{k_n-M'<i\le k_n}d^{(c)}_{n,i,\text{\tiny TV}}(c_rp_{n,i}v_n/q_n)=1,\quad\forall r\ge 1.
\]
Set $j_0=0$. Inductively, we may select positive integers $j_r,l_r$ satisfying
\[
 j_r>j_{r-1},\quad k_{j_r}-M'<l_r\le k_{j_r}
\]
such that
\begin{equation}\label{eq-jrJr}
 d^{(c)}_{j_r,l_r,\text{\tiny TV}}\left(\frac{c_rp_{j_r,l_r}v_{j_r}}{q_{j_r}}\right)\ge 1-2^{-r}.
\end{equation}
For the family $(\mathcal{X}_{j_r,l_r},H_{j_r,l_r,t},\pi_{j_r,l_r})_{r=1}^\infty$, since the maximum total variation is non-increasing, (\ref{eq-jrJr}) implies
\[
 \liminf_{r\ra\infty}d^{(c)}_{j_r,l_r,\text{\tiny TV}}\left(\frac{cp_{j_r,l_r}v_{j_r}}{q_{j_r}}\right)\ge
 \liminf_{r\ra\infty}d^{(c)}_{j_r,l_r,\text{\tiny TV}}\left(\frac{c_rp_{j_r,l_r}v_{j_r}}{q_{j_r}}\right)=1,\quad\forall c\in(0,1).
\]
while (\ref{eq-cpv/q}) yields
\[
 \lim_{r\ra\infty}d^{(c)}_{j_r,l_r,\text{\tiny TV}}\left(\frac{cp_{j_r,l_r}v_{j_r}}{q_{j_r}}\right)=0,\quad\forall c>1.
\]
This proves that the subfamily $(\mathcal{X}_{j_r,l_r},H_{j_r,l_r,t},\pi_{j_r,l_r})_{r=1}^\infty$ has a cutoff with cutoff time $p_{j_r,l_r}v_{j_r}/q_{j_r}$.
\end{proof}

\subsection{Cutoffs for some type of product chains}\label{ss-sar}
In this subsection, we consider families in Theorem \ref{t-main} and provide respectively necessary and sufficient conditions for their cutoffs in the total variation and in the Hellinger distance. For convenience, we recall the following notations
\begin{equation}\label{eq-fp2}
 \mathcal{P}=(p_n)_{n=1}^\infty, \quad\mathcal{F}=(\mathcal{X}_n,K_n,\pi_n)_{n=1}^\infty,
 \quad\mathcal{F}^\mathcal{P}=(\mathcal{Y}_n,L_n,\nu_n)_{n=1}^\infty,
\end{equation}
where $p_n>0$, $(\mathcal{X}_n,K_n,\pi_n)$ is an irreducible Markov chain and $(\mathcal{Y}_n,L_n,\nu_n)$ is the product of $(\mathcal{X}_i,K_i,\pi_i)_{i=1}^n$ according to the probability vector $(p_1/q_n,...,p_n/q_n)$, where $q_n=\sum_{i=1}^np_i$. For any sequence of positive integers $\xi=(\xi_n)_{n=1}^\infty$, we define $\mathcal{F}_\xi=(\mathcal{X}_{\xi_n},K_{\xi_n},\pi_{\xi_n})_{n=1}^\infty$ and $\mathcal{P}_\xi=(p_{\xi_n})_{n=1}^\infty$ and write $\mathcal{F}^{\mathcal{P},\xi}$ for $(\mathcal{F}_\xi)^{\mathcal{P}_\xi}$. Note that $\mathcal{F}^{\mathcal{P},\xi}$ is different from $(\mathcal{F}^\mathcal{P})_\xi$. As before, we use $\mathcal{F}^{\mathcal{P},\xi}_c$ to denote the family of continuous time Markov chains associated with $\mathcal{F}^{\mathcal{P},\xi}$. In what follows, we make some extension of Theorem \ref{t-maxcuttv} and, first of all, introduce a key technique to compare the cutoffs of $\mathcal{F}_c$ and $\mathcal{F}^\mathcal{P}_c$.

\begin{prop}\label{p-prodseqtv}
Let $\mathcal{F}^\mathcal{P}$ be the family in \textnormal{(\ref{eq-fp2})}, $\varrho_H=2\varrho_{\textnormal{\tiny TV}}=2$ and, for $*\in\{\textnormal{\tiny TV},H\}$, let $0<\epsilon_n<1/(2\varrho_*)$ be a sequence satisfying $\inf_n\epsilon_n>0$. For $n\ge 1$, let $T_{n,*}^{(c)}(\cdot)$ be the mixing time of the $n$th chain in $\mathcal{F}_c$ and set $q_n=\sum_{i=1}^np_i$ and
\[
 s_n=\max_{1\le i\le n}\frac{T^{(c)}_{i,*}(\epsilon_i)}{p_i},\quad R(c)=\lim_{m\ra\infty}\limsup_{n\ra\infty}\sum_{i=1}^{n-m}
 (2\varrho_*\epsilon_i)^{c\varrho_*s_np_i/T_{i,*}^{(c)}(\epsilon_i)}.
\]
Given any increasing sequence of positive integers $\xi=(\xi_n)_{n=1}^\infty$, set
\[
 s_n^{(\xi)}=\max_{1\le i\le n}\frac{T^{(c)}_{\xi_i,*}(\epsilon_{\xi_i})}{p_{\xi_i}},\quad R^{(\xi)}(c)=\lim_{m\ra\infty}\limsup_{n\ra\infty}\sum_{i=1}^{n-m}
 (2\varrho_*\epsilon_{\xi_i})^{c\varrho_*s_n^{(\xi)}p_{\xi_i}/T_{\xi_i,*}^{(c)}
 (\epsilon_{\xi_i})}.
\]
Then, for $*\in\{\textnormal{\tiny TV},H\}$,
\begin{itemize}
\item[(1)] If $R(c)=0$ for all $c>1$ and $\mathcal{F}_c$ has a cutoff with cutoff time $t_n$, then $\mathcal{F}^{\mathcal{P}}_c$ has a cutoff with cutoff time $q_ns_n$ and $s_n\sim\max\{t_i/p_i:1\le i\le n\}$.

\item[(2)] If $R(c)=0$ for some $c\in(0,1)$ and $\mathcal{F}^{\mathcal{P}}_c$ has a cutoff, then there is an increasing sequence of positive integers $\xi=(\xi_n)_{n=1}^\infty$ such that $(\mathcal{F}_\xi)_c$ has a cutoff.

\item[(3)] If, for any increasing sequence of positive integers $\xi$, $R^{(\xi)}(c)=0$ for some $c\in(0,1)$ and $\mathcal{F}^{\mathcal{P},\xi}_c$ has a cutoff, then $\mathcal{F}_c$ has a cutoff.
\end{itemize}
\end{prop}

\begin{proof}[Proof of Proposition \ref{p-prodseqtv}]
Referring to Theorem \ref{t-maxcuttv}, the following replacement,
\[
 k_n=n,\quad \epsilon_{n,i}=\epsilon_i,\quad t_{n,i}=t_i,\quad p_{n,i}=p_i,\quad (\mathcal{X}_{n,i},K_{n,i},\pi_{n,i})=(\mathcal{X}_i,K_i,\pi_i),
\]
leads to
\[
 (\mathcal{X}_{n,k_n-m},K_{n,k_n-m},\pi_{n,k_n-m})
 =(\mathcal{X}_{n-m},K_{n-m},\pi_{n-m}).
\]
Clearly, the notations of $\mathcal{F}^\mathcal{P}$ and $R(c)$ are consistent in Theorem \ref{t-maxcuttv} and Proposition \ref{p-prodseqtv}. As a result, (1) is given by Theorem \ref{t-maxcuttv}(1), while Theorem \ref{t-maxcuttv}(2) provides a sequence of positive integers $J$ tending to infinity such that $(\mathcal{F}_J)_c$ has a cutoff. Selecting $\xi$ as an increasing subsequence of $J$ yields (2). For (3), to show the cutoff of $\mathcal{F}_c$, it is equivalent to prove that any subfamily of $\mathcal{F}_c$ has a further subfamily that presents a cutoff. (See, for instance, \cite{CSal10} for a reference.) Let $\xi$ be an increasing sequence of positive integers. As a consequence of (2), since $\mathcal{F}^{\mathcal{P},\xi}_c$ has a cutoff and $R^{(\xi)}(c)=0$ for some $c\in(0,1)$, there is a subfamily of $(\mathcal{F}_\xi)_c$ that presents a cutoff, as desired.
\end{proof}

\begin{rem}\label{r-Rxi}
Let $N>0$ and $\xi=(\xi_n)_{n=1}^\infty$ be an increasing sequence of positive integers. Referring to the setting in Proposition \ref{p-prodseqtv}, if $s_n=T^{(c)}_{n,*}(\epsilon_n)/p_n$ for $n\ge N$, then $s_n^{(\xi)}=s_{\xi_n}$ for $\xi_n\ge N$. This implies
\[
  \sum_{i=1}^{n-m}(2\varrho_*\epsilon_{\xi_i})^{c\varrho_*s_n^{(\xi)}p_{\xi_i}/T^{(c)}_{\xi_i,*}(\epsilon_{\xi_i})}\le
  \sum_{i=1}^{\xi_n-m}(2\varrho_*\epsilon_i)^{c\varrho_*s_{\xi_n}p_i/T^{(c)}_{i,*}(\epsilon_i)},\quad\forall \xi_n\ge N,
\]
which leads to $R^{(\xi)}(c)\le R(c)$.
\end{rem}

The following is the main theorem in this subsection, which provides criteria to determine cutoffs for $\mathcal{F}^\mathcal{P}_c$.

\begin{thm}\label{t-prodseqtv}
Let $\mathcal{F}^\mathcal{P}$ be the family in \textnormal{(\ref{eq-fp2})}, $T_{n,*}^{(c)}$ be the mixing time of the $n$th chain in $\mathcal{F}_c$ and $\varrho_H=2\varrho_{\textnormal{\tiny TV}}=2$. Assume that, for $*\in\{\textnormal{\tiny TV},H\}$, there are constants $c_0\in(0,1)$, $N>0$ and a sequence $(\epsilon_n)_{n=1}^\infty$ satisfying $0<\inf_n\epsilon_n\le\sup_n\epsilon_n<1/(2\varrho_*)$ such that
\begin{equation}\label{eq-prod1}
 \max_{1\le i\le n}\frac{T_{i,*}^{(c)}(\epsilon_i)}{p_i}=\frac{T_{n,*}^{(c)}(\epsilon_n)}{p_n},\quad\forall n\ge N,
\end{equation}
and
\begin{equation}\label{eq-prod2}
 \lim_{m\ra\infty}\limsup_{n\ra\infty}\sum_{i=1}^{n-m}\left(2\varrho_*\sup_{n\ge 1}\epsilon_n\right)^{c_0\varrho_*T_{n,*}^{(c)}(\epsilon_n)p_i/
 (T_{i,*}^{(c)}(\epsilon_i)p_n)}=0.
\end{equation}
Then, for $*\in\{\textnormal{\tiny TV},H\}$,
\begin{itemize}
\item[(1)] $\mathcal{F}_c$ has a cutoff if and only if, for any increasing sequence of positive integers $\xi$, $\mathcal{F}^{\mathcal{P},\xi}_c$ has a cutoff. In particular, if $\mathcal{F}_c$ has a cutoff, then $\mathcal{F}^\mathcal{P}_c$ has a cutoff.

\item[(2)] No subfamily of $\mathcal{F}_c$ has a cutoff if and only if, for any increasing sequence of positive integers $\xi$, $\mathcal{F}^{\mathcal{P},\xi}_c$ has no cutoff. In particular, if $\mathcal{F}_c$ has no subfamily presenting cutoff, then $\mathcal{F}^\mathcal{P}_c$ has no cutoff.
\end{itemize}
Further, if $\mathcal{F}_c$ has cutoff time $t_n$, then $\mathcal{F}^{\mathcal{P}}_c$ has cutoff time $q_nt_n/p_n$, where $q_n=p_1+\cdots+p_n$.
\end{thm}

\begin{proof}
Let $\xi$ be an increasing sequence of positive integers, $c_0,N$ be the constants in Theorem \ref{t-prodseqtv} and $s_n,R(c),R^{(\xi)}(c)$ be as in Proposition \ref{p-prodseqtv}. By (\ref{eq-prod1}), one has $s_n=T^{(c)}_{n,*}(\epsilon_n)/p_n$ for $n\ge N$ and, by (\ref{eq-prod2}) and Remark \ref{r-Rxi}, this implies
\[
 R^{(\xi)}(c)\le R(c)=0,\quad\forall c\ge c_0,\,\xi.
\]
For (1), based on the above observation and Proposition \ref{p-prodseqtv}(3), it is obvious that if $\mathcal{F}^{\mathcal{P},\xi}_c$ has a cutoff for any $\xi$, then $\mathcal{F}_c$ has a cutoff. Conversely, if $\mathcal{F}_c$ has a cutoff, then $(\mathcal{F}_\xi)_c$ has a cutoff for all $\xi$ and, as a consequence of Proposition \ref{p-prodseqtv}(1), $\mathcal{F}^{\mathcal{P},\xi}_c$ has a cutoff. Note that, when $\mathcal{F}_c$ has a cutoff, the desired cutoff time for $\mathcal{F}^\mathcal{P}_c$ is given by Proposition \ref{p-prodseqtv}(1).

Next, we discuss (2). By Proposition \ref{p-prodseqtv}(1), if $(\mathcal{F}_\xi)_c$ has a cutoff, then $\mathcal{F}^{\mathcal{P},\xi}_c$ has a cutoff. Conversely, by Proposition \ref{p-prodseqtv}(2), if $\mathcal{F}^{\mathcal{P},\xi}_c$ has a cutoff, then there is a subsequence of $\xi$, say $\xi'$, such that $(\mathcal{F}_{\xi'})_c$ has a cutoff. This proves (2).
\end{proof}

\begin{rem}
We would like to point out the non-consistency of cutoffs for $\mathcal{F}_c$ and $\mathcal{F}^{\mathcal{P}}_c$ and illustrate this observation with examples in Subsection \ref{ss-noncnsst}.
\end{rem}

\begin{proof}[Proof of Theorem \ref{t-main}]
The proofs for the total variation and the Hellinger distance are similar and we deal with the case of the total variation. Set $\epsilon:=2\sup_n\epsilon_n$ and $C:=\sup_n|C_n|$. Since $D_n$ is nondecreasing for $n$ large enough and $D_n\ra\infty$, (\ref{eq-prod1}) holds. By Theorem \ref{t-prodseqtv}, it remain to show that there is $c_0\in(0,1)$ such that
\[
 \lim_{m\ra\infty}\limsup_{n\ra\infty}\sum_{i=1}^{n-m}\epsilon^{c_0\exp\{D_n-D_i\}}=0.
\]
Since $B_n$ is nondecreasing and $|C_n|\le C$, one has $D_n-D_i\ge A_nn-A_ii-2C$. Note that, if $A_n$ is nondecreasing, then $A_nn-A_ii\ge A_1(n-i)$. If $A_n=A+O(1/n)$ for some $A>0$, then
\[
 A_nn-A_ii=A(n-i)-|A_n-A|n-|A_i-A|i\ge A(n-i)-2A',
\]
where $A'=\sup_nn|A_n-A|$. As a result, we obtain $e^{D_n-D_i}\ge (A_1\wedge A)e^{-2(A'+C)}(n-i)$ and, by setting $\epsilon'=\exp\{c_0(A_1\wedge A)e^{-2(A'+C)}\log\epsilon\}\in(0,1)$, this leads to
\[
 \limsup_{n\ra\infty}\sum_{i=1}^{n-m}\epsilon^{c_0\exp\{D_n-D_i\}}\le \sum_{j=m}^\infty (\epsilon')^j=\frac{(\epsilon')^m}{1-\epsilon'}\ra 0,\quad\text{as }m\ra\infty,
\]
for all $c_0>0$.
\end{proof}

\section{Examples}\label{s-examples}

In this section, we consider practical examples for families in Theorem \ref{t-main}, which are exactly families in Subsection \ref{ss-sar}, and determine their cutoffs.

\subsection{Products of two-state chains}

Let $(\mu,\mathcal{X},K,\pi)$ be an irreducible Markov chain and $(\mu,\mathcal{X},H_t,\pi)$ be the associated continuous time chain. Define the $L^2$-distance and the $L^2$-mixing time of $(\mu,\mathcal{X},H_t,\pi)$ by
\[
 d_2^{(c)}(\mu,t)=\left(\sum_{x\in\mathcal{X}}\left|\frac{\mu H_t(x)}{\pi(x)}-1\right|^2\pi(x)\right)^{1/2},\quad
 T_2^{(c)}(\mu,\epsilon)=\min\{t\ge 0|d_2^{(c)}(\mu,t)\le\epsilon\}.
\]
For two-state chains, we have the following precise computations.

\begin{lem}\label{l-2state}
Let $(\mathcal{X},H_t,\pi)$ be a continuous time Markov chain associated with $(\mathcal{X},K,\pi)$, where
\[
 \mathcal{X}=\{0,1\},\quad K=\left(\begin{array}{cc}1-\alpha&\alpha\\\beta&1-\beta\end{array}\right),\quad
 \pi=\left(\frac{\beta}{\alpha+\beta},\frac{\alpha}{\alpha+\beta}\right).
\]
For $t\ge 0$, one has
\[
 d_2^{(c)}(0,t)^2=\frac{\alpha}{\beta}e^{-2(\alpha+\beta)t},\quad d_H^{(c)}(0,t)^2=\frac{d_2^{(c)}(0,t)^2}{r(t)},
\]
where $r(t)=[1+A(t)][1+B(t)][A(t)+B(t)]$ and
\[
 A(t)=\sqrt{1+\frac{\alpha}{\beta}e^{-(\alpha+\beta)t}},\quad B(t)=\sqrt{1-e^{-(\alpha+\beta)t}}.
\]
In particular,
\[
 \frac{d_2^{(c)}(0,t)^2}{4[2+(\alpha/\beta)e^{-(\alpha+\beta)t}]}\le d_H^{(c)}(0,t)^2\le \frac{d_2^{(c)}(0,t)^2}{2+(\alpha/\beta)e^{-(\alpha+\beta)t}}.
\]
\end{lem}
\begin{proof}
The $L^2$-distance is given by the spectral information of $K$, while the Hellinger distance follows immediately from (\ref{eq-2phd}).
\end{proof}

Clearly, one can see from the above lemma that the Hellinger distance and the $L^2$-distance of two-state chains are comparable with each other.

Next, we consider the cutoff in the $L^2$-distance. A family of continuous time Markov chains $\mathcal{F}_c=(\mu_n,\mathcal{X}_n,H_{n,t},\pi_n)_{n=1}^\infty$ is said to present a $L^2$-cutoff if there is a sequence $t_n>0$ such that
\[
 \lim_{n\ra\infty}d_{n,2}^{(c)}(\mu_n,at_n)=\begin{cases}0&\text{for }a\in(1,\infty),\\\infty&\text{for }a\in(0,1).\end{cases}
\]
$\mathcal{F}_c$ is said to present a $(t_n,b_n)$ $L^2$-cutoff if $t_n>0$, $b_n>0$, $b_n=o(t_n)$ and
\[
 \lim_{c\ra\infty}\limsup_{n\ra\infty}d_{n,2}^{(c)}(\mu_n,t_n+cb_n)=0,\quad
 \lim_{c\ra-\infty}\liminf_{n\ra\infty}d_{n,2}^{(c)}(\mu_n,t_n+cb_n)=\infty.
\]
For product chains, Chen, Hsu and Sheu declare the following observation in \cite{CHS16}.

\begin{lem}{\rm(\cite[Proposition 4.1]{CHS16})}\label{l-l2prod}
Let $\mathcal{F}$ and $\mathcal{F}^\mathcal{P}$ be families in \textnormal{(\ref{eq-fp2})} with initial distributions $(\mu_n)_{n=1}^\infty$ and $(\sigma_n)_{n=1}^\infty$, where $\sigma_n=\mu_1\times\cdots\times\mu_n$.
\begin{itemize}
\item[(1)] $\mathcal{F}^\mathcal{P}_c$ has a $L^2$-cutoff with cutoff time $t_n$ if and only if
\[
 \lim_{n\ra\infty}\sum_{i=1}^nd_{i,2}^{(c)}(\mu_i,ap_it_n/q_n)=\begin{cases}
 0&\text{for }a>1,\\\infty&\text{for }a\in(0,1).\end{cases}
\]

\item[(2)] $\mathcal{F}^\mathcal{P}_c$ has a $(t_n,b_n)$ $L^2$-cutoff if and only if $t_n>0$, $b_n>0$, $b_n=o(t_n)$ and
\[
 \lim_{c\ra\infty}\limsup_{n\ra\infty}\sum_{i=1}^nd_{i,2}^{(c)}
 (\mu_i,(p_i/q_n)(t_n+cb_n))^2=0,
\]
and
\[
 \lim_{c\ra-\infty}\liminf_{n\ra\infty}\sum_{i=1}^nd_{i,2}^{(c)}
 (\mu_i,(p_i/q_n)(t_n+cb_n))^2=\infty.
\]
\end{itemize}
\end{lem}

As a consequence of Lemmas \ref{l-2state}-\ref{l-l2prod}, Proposition \ref{p-comp0} and Theorem \ref{t-prodcutoffhd}, we achieve the following proposition.

\begin{prop}\label{p-l1l2comp}
Let $\mathcal{F}^\mathcal{P}$ be the family in \textnormal{(\ref{eq-fp2})} with
\[
 \mathcal{X}_n=\{0,1\},\quad K_n=\left(\begin{array}{cc}1-\alpha_n&\alpha_n\\\beta_n&1-\beta_n\end{array}\right),\quad
 \pi_n=\left(\frac{\beta_n}{\alpha_n+\beta_n},\frac{\alpha_n}{\alpha_n+\beta_n}\right).
\]
Suppose the $n$th chain in $\mathcal{F}^\mathcal{P}$ starts at $\mathbf{0}$, the zero vector in $\mathcal{Y}_n$, and assume that $\sup_n\{\alpha_n/\beta_n\}<\infty$. Then,
\begin{itemize}
\item[(1)] $\mathcal{F}^\mathcal{P}_c$ has a total variation cutoff if and only if $\mathcal{F}^\mathcal{P}_c$ has a $L^2$-cutoff. Furthermore, $T^{(c)}_{n,\textnormal{\tiny TV}}(\mathbf{0},\epsilon)\sim T_{n,2}^{(c)}(\mathbf{0},\delta)$ for all $\epsilon\in(0,1)$ and $\delta>0$.

\item[(2)] $\mathcal{F}^\mathcal{P}_c$ has a $(t_n,b_n)$ total variation cutoff if and only if $\mathcal{F}^\mathcal{P}_c$ has a $(t_n,b_n)$ $L^2$-cutoff.
\end{itemize}
\end{prop}
\begin{proof}
Note that, by Proposition \ref{p-comp0}, it suffices to show the equivalence of cutoffs in the Hellinger distance and the $L^2$-distance. Set $r=\sup_n\{\alpha_n/\beta_n\}$. By Lemma \ref{l-2state}, one has
\begin{equation}\label{eq-hl2comp}
 \frac{1}{4(2+r)}\sum_{i=1}^nd_{i,2}^{(c)}(0,p_it/q_n)^2\le
 \sum_{i=1}^nd_{i,H}^{(c)}(0,p_it/q_n)^2\le\frac{1}{2}\sum_{i=1}^nd_{i,2}^{(c)}(0,p_it/q_n)^2.
\end{equation}
The proof is based on the above inequalities. We first consider (1) and set, for $a>0$,
\[
 D_2(a)=\lim_{n\ra\infty}\sum_{i=1}^nd_{i,2}^{(c)}(0,ap_it_n/q_n)^2,\quad
 D_H(a)=\lim_{n\ra\infty}\frac{\sum_{i=1}^nd_{i,H}^{(c)}(0,ap_it_n/q_n)^2}{1-\max\limits_{1\le i\le n}d_{i,H}^{(c)}(0,ap_it_n/q_n)^2}.
\]
By (\ref{eq-hl2comp}), one has
\[
 D_2(a)=0\quad\Lra\quad D_H(a)=0,\quad D_2(a)=\infty\quad\Ra\quad D_H(a)=\infty.
\]
As a result of Theorem \ref{t-prodcutoffhd} and Lemma \ref{l-l2prod}, if $\mathcal{F}^\mathcal{P}_c$ has a $L^2$-cutoff with cutoff time $t_n$, then $\mathcal{F}^\mathcal{P}_c$ has a cutoff in the Hellinger distance with cutoff time $t_n$. Further, if $\mathcal{F}^\mathcal{P}_c$ has a cutoff in the Hellinger distance with cutoff time $t_n$, then $D_2(a)=0$ for $a>1$. To finish the proof of (1), it remains to show that $D_2(a)=\infty$ for $0<a<1$. Assume the inverse that there are $a_0\in(0,1)$ and a subsequence $\xi=(\xi_n)_{n=1}^\infty$ such that
\[
 \lim_{n\ra\infty}\sum_{i=1}^{\xi_n}d_{i,2}^{(c)}(0,a_0p_it_{\xi_n}/q_{\xi_n})^2<\infty,
\]
and set
\[
 \widetilde{D}_2(a)=\limsup_{n\ra\infty}
 \sum_{i=1}^{\xi_n}d_{i,2}^{(c)}(0,ap_it_{\xi_n}/q_{\xi_n})^2,\quad\forall a>a_0.
\]
Since $D_2(a)=0$ for $a>1$, one has $\widetilde{D}_2(a)=0$ for $a>1$. It is easy to see from Lemma \ref{l-2state} that the summation defining $\widetilde{D}_2$ is a linear combination of exponential functions with positive coefficients. As a consequence of Lemma \ref{l-analytic3}, $\widetilde{D}_2(a)=0$ for $a>a_0$ and, by (\ref{eq-hl2comp}), this leads to
\begin{equation}\label{eq-cnt1}
 \lim_{n\ra\infty}\frac{\sum_{i=1}^{\xi_n}d_{i,H}^{(c)}(0,ap_it_{\xi_n}/q_{\xi_n})^2}
 {1-\max\limits_{1\le i\le \xi_n}d_{i,H}^{(c)}(0,ap_it_{\xi_n}/q_{\xi_n})^2}=0
 ,\quad\forall a>a_0.
\end{equation}
However, by Theorem \ref{t-prodcutoffhd}, the cutoff of $\mathcal{F}^\mathcal{P}_c$in the Hellinger distance with cutoff time $t_n$ yields $D_H(a)=\infty$ for $0<a<1$, which contradicts (\ref{eq-cnt1}).

Next, we consider (2). In a similar reasoning, one can show that a $(t_n,b_n)$ $L^2$-cutoff implies a $(t_n,b_n)$ cutoff in the Hellinger distance. Further, a $(t_n,b_n)$ cutoff in the Hellinger distance implies
\[
 \lim_{c\ra\infty}\limsup_{n\ra\infty}\sum_{i=1}^nd_{i,2}^{(c)}(0,(t_n+cb_n)p_i/q_n)^2=0.
\]
To finish the proof of (2), one needs to show that
\[
 \lim_{c\ra-\infty}\liminf_{n\ra\infty}\sum_{i=1}^nd_{i,2}^{(c)}(0,(t_n+cb_n)p_i/q_n)^2
 =\infty,
\]
when $\mathcal{F}^\mathcal{P}_c$ has a $(t_n,b_n)$ cutoff in the Hellinger distance.
Assume the inverse that there are $c_n\ra\infty$ and a subsequence $\xi=(\xi_n)_{n=1}^\infty$ such that
\begin{equation}\label{eq-cnt2}
 t_{\xi_n}/(c_nb_{\xi_n})\ra\infty,\quad\limsup_{n\ra\infty}
 \sum_{i=1}^{\xi_n}d_{i,2}^{(c)}(0,(t_{\xi_n}-c_nb_{\xi_n})p_i/q_{\xi_n})^2<\infty,
\end{equation}
and set
\[
 \overline{D}_2(a)=\limsup_{n\ra\infty}
 \sum_{i=1}^{\xi_n}d_{i,2}^{(c)}(0,(t_{\xi_n}+ac_nb_{\xi_n})p_i/q_{\xi_n})^2.
\]
By the former of (\ref{eq-cnt2}), it is clear that $\overline{D}_2(a)$ is defined for $a\in\mathbb{R}$. Since $(\mathcal{F}^\mathcal{P}_c)_\xi$ has a $(t_{\xi_n},b_{\xi_n})$ cutoff in the Hellinger distance and $c_n\ra\infty$, one has
\begin{equation}\label{eq-cnt3}
 \overline{D}_2(1)=0,\quad \lim_{n\ra\infty}\frac{\sum_{i=1}^{\xi_n}d_{i,H}^{(c)}(0,(t_{\xi_n}+ac_nb_{\xi_n}))^2}
 {1-\max\limits_{1\le i\le \xi_n}d_{i,H}^{(c)}(0,(t_{\xi_n}+ac_nb_{\xi_n}))^2}=\infty,
 \quad\forall a<0.
\end{equation}
By Lemma \ref{l-analytic3}, the latter of (\ref{eq-cnt2}) and the former of (\ref{eq-cnt3}) imply $\overline{D}_2(a)=0$ for $a>-1$. Consequently, (\ref{eq-hl2comp}) yields
\[
 \lim_{n\ra\infty}\frac{\sum_{i=1}^{\xi_n}d_{i,H}^{(c)}(0,(t_{\xi_n}+ac_nb_{\xi_n}))^2}
 {1-\max\limits_{1\le i\le \xi_n}d_{i,H}^{(c)}(0,(t_{\xi_n}+ac_nb_{\xi_n}))^2}=0,
 \quad\forall a>-1,
\]
which contradicts the latter of (\ref{eq-cnt3}).
\end{proof}

\begin{rem}
In fact, one may derive a similar version of Lemma \ref{l-2state} to compare the total variation and the $L^2$-distance of two-state
chains. However, this is not sufficient to prove Proposition \ref{p-l1l2comp} due to the lack of a similar version of Theorem \ref{t-prodcutoffhd} in the total variation.
\end{rem}

Concerning families of reversible Markov chains, Chen and Saloff-Coste obtain an equivalent condition for the $L^2$-cutoff in \cite{CSal10}, while Chen, Hsu and Sheu polish their result in \cite{CHS16}. The following theorem is a combination of \cite[Theorem 4.3]{CHS16} and Proposition \ref{p-l1l2comp}.

\begin{thm}\label{t-l1cutoff}
Let $\mathcal{F}^\mathcal{P}$ be the family in Proposition \ref{p-l1l2comp} and assume $\inf_n\alpha_n\wedge\beta_n>0$ and $p_n\le p_{n+1}$. Then, $\mathcal{F}^\mathcal{P}_c$ has a total variation cutoff if and only if
\begin{equation}\label{eq-l2cut}
 \sup_{n\ge 1}\frac{\log(1+n)}{p_n}=\infty.
\end{equation}
Moreover, if \textnormal{(\ref{eq-l2cut})} holds and $p_n(\alpha_n+\beta_n)$ is increasing, then $\mathcal{F}^\mathcal{P}_c$ has a $(t_n,b_n)$ total variation cutoff, where
\begin{equation}\label{eq-tn_2state}
 t_n=q_n\max_{1\le j\le n}\frac{\log(1+j)}{2p_j(\alpha_j+\beta_j)},\quad b_n=\sqrt{t_nq_n},\quad q_n=\sum_{i=1}^np_i.
\end{equation}
\end{thm}

\begin{rem}
Let $\mathcal{F}^\mathcal{P}$ be the family in Proposition \ref{p-l1l2comp} satisfying
\[
  \inf_{n\ge 1}\alpha_n\wedge\beta_n>0,\quad p_n\le p_{n+1},\quad p_n(\alpha_n+\beta_n)\le p_{n+1}(\alpha_{n+1}+\beta_{n+1})
\]
and let $T_{n,2}^{(c)}$ and $T_{n,\text{\tiny TV}}^{(c)}$ be the mixing times of the $n$th chain in $\mathcal{F}^\mathcal{P}_c$ in the $L^2$-distance and in the total variation.
In \cite{CHS16}, Chen, Hsu and Sheu show that there is $\epsilon_0>0$ such that $T_{n,2}^{(c)}(\mathbf{0},\epsilon)\asymp t_n$ for $\epsilon\in(0,\epsilon_0)$, where $t_n$ is the constant in (\ref{eq-tn_2state}). By using \cite[Proposition 4.1]{CHS16}, Proposition \ref{p-prodmixing} and Lemma \ref{l-2state}, one may select $0<\epsilon_1<\epsilon_0$ such that $T_{n,\text{\tiny TV}}^{(c)}(\mathbf{0},\epsilon)\asymp t_n$ for $\epsilon\in(0,\epsilon_1)$. Note that the spectral gap $\lambda_n$ of the $n$th chain in $\mathcal{F}^\mathcal{P}_c$, which is the smallest nonzero eigenvalue of $I-L_n$, is equal to $p_1(\alpha_1+\beta_1)/q_n\asymp 1/q_n$. As a consequence of Theorem \ref{t-l1cutoff}, we obtain, for $\epsilon\in(0,\epsilon_1)$,
\[
 \mathcal{F}^\mathcal{P}_c\text{ has a total variation cutoff}\quad\Lra\quad T_{n,\text{\tiny TV}}^{(c)}(\mathbf{0},\epsilon)\lambda_n\ra \infty,
\]
and
\[
 \mathcal{F}^\mathcal{P}_c\text{ has a $L^2$-cutoff}\quad\Lra\quad T_{n,2}(\mathbf{0},\epsilon)^{(c)}\lambda_n\ra \infty.
\]
Since Peres conjectured that a cutoff exists if and only if the product of the mixing time and spectral gap tends to infinity, the above equivalences confirm this hypothesis for $\mathcal{F}^\mathcal{P}_c$ in the total variation and in the $L^2$-distance.
\end{rem}

\subsection{Counterexamples to the consistency of cutoffs}\label{ss-noncnsst}

Referring to the setting in (\ref{eq-fp2}),
we give the proof of Theorem \ref{thm:counterexs} in this subsection by providing two examples, which respectively displays that none of cutoffs for $\mathcal{F}_c$ and $\mathcal{F}_c^\mathcal{P}$ implies the other. As cutoffs in the total variation and Hellinger distance are identified by Proposition \ref{p-comp0}, we will discuss those examples in either convenient way.

\subsubsection{$\mathcal{F}_c$ has no cutoff but $\mathcal{F}^\mathcal{P}_c$ presents one}

Consider the following setting. For $i=1,2$, let $\mathcal{F}^{(i)}=(\mathcal{X}_n^{(i)},K_n^{(i)},\pi^{(i)}_n)_{n=1}^\infty$ be a family of irreducible Markov chains, where $\mathcal{X}_n^{(1)}=\mathcal{X}_n^{(2)}=\{0,1,...,n\}$ and
\[
 K_n^{(1)}(j,j+1)=\frac{n-j}{n},\quad K_n^{(1)}(j+1,j)=\frac{j+1}{n},\quad\forall 0\le j<n,
\]
and
\[
 K_n^{(2)}(j,j+1)=K_n^{(2)}(j+1,j)=K_n^{(2)}(0,0)=K_n^{(2)}(n,n)=\frac{1}{2},\quad\forall 0\le j<n.
\]
It is easy to check that $\pi_n^{(1)}(j)=2^{-n}\binom{n}{j}$ and $\pi_n^{(2)}(j)=(n+1)^{-1}$. We use the notations of $d_{n,\text{\tiny TV}}^{(i,c)}$ and $T_{n,\text{\tiny TV}}^{(i,c)}$ to denote the total variation and the corresponding mixing time of the $n$th chain in $\mathcal{F}_c^{(i)}$. It is well-studied that $\mathcal{F}_c^{(1)}$ has a total variation cutoff with cutoff time $\frac{1}{4}n\log n$; $\mathcal{F}_c^{(2)}$
has no cutoff in the total variation but the mixing time satisfies $T_{n,\text{\tiny TV}}^{(2,c)}(\epsilon)\asymp n^2$ for all $\epsilon\in(0,1)$. Let $\mathcal{F}=(\mathcal{X}_n,K_n,\pi_n)_{n=1}^\infty$ be the mixed family of $\mathcal{F}^{(1)}$ and $\mathcal{F}^{(2)}$ in the way that
\[
 (\mathcal{X}_{2n-1},K_{2n-1},\pi_{2n-1})=(\mathcal{X}_n^{(1)},K_n^{(1)},\pi_n^{(1)}),\quad
 (\mathcal{X}_{2n},K_{2n},\pi_{2n})=(\mathcal{X}_n^{(2)},K_n^{(2)},\pi_n^{(2)}).
\]
Since $\mathcal{F}_c^{(2)}$ has no cutoff, $\mathcal{F}_c$ has no cutoff either.

To see a product chain of $\mathcal{F}$ with cutoff, we consider the following sequence
\[
 p_{2n-1}=r^{n-1},\quad p_{2n}=1,\quad \forall n\ge 1,
\]
with $r\in(0,1)$ and write $\mathcal{P}=(p_n)_{n=1}^\infty$. Let $\mathcal{P}_1=(p_{2n-1})_{n=1}^\infty$, $\mathcal{P}_2=(p_{2n})_{n=1}^\infty$ and set
\[
 q_n=\sum_{i=1}^np_i,\quad q^{(1)}_n=\sum_{i=1}^np_{2i-1},\quad q^{(2)}_n=\sum_{i=1}^np_{2i}.
\]
It is obvious that $q_{2n-1}=q^{(1)}_n+q^{(2)}_{n-1}$ and $q_{2n}=q^{(1)}_n+q^{(2)}_n$. To check the existence of cutoff for $\mathcal{F}^\mathcal{P}_c$, we need the following notations. For $n\ge 1$, let $d_{n,\text{\tiny TV}}^{(c)}$ and $d_{n,\text{\tiny TV}}^{(\mathcal{P}_i,c)}$ be the total variation of the $n$th chains in $\mathcal{F}_c^\mathcal{P}$ and $(\mathcal{F}^{(i)})_c^{\mathcal{P}_i}$, and let $T_{n,\text{\tiny TV}}^{(c)}$ and $T_{n,\text{\tiny TV}}^{(\mathcal{P}_i,c)}$ be the corresponding mixing times. As a consequence of Lemma \ref{l-prodmixing}, we have
\begin{equation}\label{eq-d2n-1}
 d_{2n-1,\text{\tiny TV}}^{(c)}(t)\begin{cases}\le d_{n,\text{\tiny TV}}^{(\mathcal{P}_1,c)}\left(\tfrac{q^{(1)}_nt}{q_{2n-1}}\right)
 +d_{n-1,\text{\tiny TV}}^{(\mathcal{P}_2,c)}\left(\tfrac{q^{(2)}_{n-1}t}{q_{2n-1}}\right)\\
 \ge \max\left\{d_{n,\text{\tiny TV}}^{(\mathcal{P}_1,c)}\left(\tfrac{q^{(1)}_nt}{q_{2n-1}}\right),
 d_{n-1,\text{\tiny TV}}^{(\mathcal{P}_2,c)}\left(\tfrac{q^{(2)}_{n-1}t}{q_{2n-1}}\right)\right\}\end{cases}
\end{equation}
and
\begin{equation}\label{eq-d2n-2}
 d_{2n,\text{\tiny TV}}^{(c)}(t)\begin{cases}\le d_{n,\text{\tiny TV}}^{(\mathcal{P}_1,c)}\left(\tfrac{q^{(1)}_nt}{q_{2n}}\right)
 +d_{n,\text{\tiny TV}}^{(\mathcal{P}_2,c)}\left(\tfrac{q^{(2)}_{n}t}{q_{2n}}\right)\\
 \ge \max\left\{d_{n,\text{\tiny TV}}^{(\mathcal{P}_1,c)}\left(\tfrac{q^{(1)}_nt}{q_{2n}}\right),
 d_{n,\text{\tiny TV}}^{(\mathcal{P}_2,c)}\left(\tfrac{q^{(2)}_{n}t}{q_{2n}}\right)\right\}\end{cases}
\end{equation}

Next, we show that $(\mathcal{F}^{(1)})_c^{\mathcal{P}_1}$  has a cutoff. Since $\mathcal{F}_c^{(1)}$ has a cutoff with cutoff time $\frac{1}{4}n\log n$, one has $T_{n,\text{\tiny TV}}^{(1,c)}(\epsilon)\sim\frac{1}{4}n\log n$ for all $\epsilon\in(0,1)$. In some computations, we obtain
\[
 \lim_{n\ra\infty}\frac{T_{n+1,\text{\tiny TV}}^{(1,c)}(\epsilon)/p_{2n+1}}
 {T_{n,\text{\tiny TV}}^{(1,c)}(\epsilon)/p_{2n-1}}=\frac{1}{r}>1,
\]
and
\[
 \log\frac{T_{n,\text{\tiny TV}}^{(1,c)}(\epsilon)}{p_{2n-1}}=\left(\log\frac{1}{r}\right)n+\log n+\log\log n+O(1).
\]
The former implies that $T_{n,\text{\tiny TV}}^{(1,c)}(\epsilon)/p_{2n-1}$ is increasing for $n$ large enough. By Theorem \ref{t-main}, $(\mathcal{F}^{(1)})^{\mathcal{P}_1}_c$ has a cutoff with cutoff time
$\frac{1}{4}q^{(1)}_nr^{1-n}n\log n\sim\frac{r}{4(1-r)}r^{-n}n\log n$.

Now, we show that $\mathcal{F}_c^\mathcal{P}$ has a cutoff with cutoff time $t_n$, where
\[
 t_{2n-1}=\frac{1}{4}q_{2n-1}r^{1-n}n\log n,\quad t_{2n}=\frac{1}{4}q_{2n}r^{1-n}n\log n.
\]
Note that $(\mathcal{F})_c^{\mathcal{P}_1}$ has a cutoff with cutoff time $t_n/q_n$. By (\ref{eq-d2n-1}) and (\ref{eq-d2n-2}), to finish the proof, it suffices to prove
\[
 \lim_{n\ra\infty}d_{n,\text{\tiny TV}}^{(\mathcal{P}_2,c)}\left(\frac{cr}{4}q_n^{(2)}r^{-n}n\log n\right)=0,\quad\forall c>1.
\]
Let $B>0$ be such that $T_{n,\text{\tiny TV}}^{(2,c)}(1/(2e))\le Bn^2$ for all $n\ge 1$. Observe that, for fixed $C>0$, $Cq_n^{(2)}r^{1-n}n\log n>Bn^2q_n^{(2)}\ge q_n^{(2)}\max\{T_{i,\text{\tiny TV}}^{(2,c)}(1/(2e))/p_{2i-1}:1\le i\le n\}$ for $n$ large enough. By Proposition \ref{p-prodmaxtv}, this implies
\begin{align}
 &\limsup_{n\ra\infty}d_{n,\text{\tiny TV}}^{(\mathcal{P}_2,c)}\left(Cq_n^{(2)}r^{1-n}n\log n\right)\notag\\
 \le&1-\exp\left\{-\limsup_{n\ra\infty}\sum_{i=1}^n
 \exp\left\{-\left\lfloor\frac{Cr^{1-n}n\log n}{T_{n,\text{\tiny TV}}^{(2),c}(1/(2e))}\right\rfloor\right\}\right\}\notag\\
 \le&1-\exp\left\{-e\cdot\limsup_{n\ra\infty}
 \sum_{i=1}^n\exp\left\{-\frac{Cr^{1-n}n\log n}{Bi^2}\right\}\right\}\notag\\
 \le&1-\exp\left\{-e\cdot\limsup_{n\ra\infty}n^{1-Cr^{1-n}/(Bn)}\right\}=0.\notag
\end{align}

\subsubsection{$\mathcal{F}_c$ presents a cutoff but $\mathcal{F}^\mathcal{P}_c$ does not}
We will use the chain in Example \ref{ex-Lacoin} to create our counterexample. First of all, we make some analysis on products of chains in (\ref{eq-LacoinK}) and result in a list of observations. As the proofs are somewhat technical, we address all of them in the appendix in order to keep our construction clear.

\begin{lem}\label{l-Lacoin1}
For $n\ge 1$ and $1\le i\le n$, let $p_{n,i}>0$ and $(\mathcal{X}_{n,i},K_{n,i},\pi_{n,i})$ be the Markov chain in \textnormal{(\ref{eq-LacoinK})} with $\beta=0$, $a_{n,i}<b_{n,i}$ and $a_{n,i}+b_{n,i}<1/2$. Consider the family $\mathcal{G}=(\mathcal{X}_n,K_n,\pi_n)_{n=1}^\infty$, where $(\mathcal{X}_n,K_n,\pi_n)$ is the product chain of $(\mathcal{X}_{n,i},K_{n,i},\pi_{n,i})_{i=1}^n$ according to the probability vector $(p_{n,i}/q_n)_{i=1}^n$ and $q_n=p_{n,1}+\cdots+p_{n,n}$. Let $d_{n,H}^{(c)}$ be the Hellinger distance of the $n$th chain in $\mathcal{G}_c$ and set $\hat{p}_n=\min\{p_{n,i}|1\le i\le n\}$.
\begin{itemize}
\item[(1)] If $\sum_{i=1}^na_{n,i}=o(1/n)$, then, for any $C>1$,
\begin{equation}\label{eq-Lacoinprecutoff}
 \lim_{n\ra\infty}d_{n,H}^{(c)}\left(C^{-1}q_nn/\hat{p}_n\right)=1,\quad
 \lim_{n\ra\infty}d_{n,H}^{(c)}\left(2Cq_nn/\hat{p}_n\right)=0.
\end{equation}

\item[(2)] Set $E_{n,\delta}=\{1\le i\le n|p_{n,i}<(1+\delta)\hat{p}_n\}$ and $B_n(\delta)=\sum_{i\in E_{n,\delta}}b_{n,i}$. If it is assumed
\begin{equation}\label{eq-Lacoin2}
 \sum_{i=1}^na_{n,i}=o\left(\frac{1}{n}\right),\quad\max_{1\le i\le n}b_{n,i}=o(1),\quad\max_{1\le i\le n}\frac{a_{n,i}}{b_{n,i}}=O\left(\frac{1}{n}\right),
\end{equation}
then, for $0<\Delta_-<\Delta<\Delta_+<1$,
\begin{equation}\label{eq-Lacoin3}
 1-e^{-B_n(\Delta_-)(1/2+o(1))}\le d_{n,H}^{(c)}\left(\frac{2q_nn}{(1+\Delta)\hat{p}_n}\right)\le 1-e^{-B_n(\Delta_+)(1+o(1))}.
\end{equation}

\end{itemize}
\end{lem}

\begin{rem}
Lemma \ref{l-Lacoin1}(1) implies that $\mathcal{G}_c$ has a total variation pre-cutoff.
\end{rem}

To build up a criterion on cutoffs from Lemma \ref{l-Lacoin1}, we introduce the following notations. Let $B_n(\delta)$ be the function in Lemma \ref{l-Lacoin1} and set, for any increasing  sequence $\xi=(\xi_n)_{n=1}^\infty$ in $\mathbb{N}$,
\begin{equation}\label{eq-Fxi}
 \overline{F}_{\xi}(\delta)=\limsup_{n\ra\infty}B_{\xi_n}(\delta),\quad
 \underline{F}_{\xi}(\delta)=\liminf_{n\ra\infty}B_{\xi_n}(\delta),
\end{equation}
and, for $c\in[0,\infty]$,
\begin{equation}\label{eq-Deltac}
 \overline{\Delta}_c(\xi):=\sup\{\delta\in(0,1)|\overline{F}_\xi(\delta)=c\},\quad
 \underline{\Delta}_c(\xi):=\sup\{\delta\in(0,1)|\underline{F}_\xi(\delta)=c\},
\end{equation}
where $\sup\emptyset:=0$ and $\inf\emptyset:=1$. If $\xi_n=n$, we simply write $\overline{F},\underline{F},\overline{\Delta}_c,\underline{\Delta}_c$  for $\overline{F}_\xi,\underline{F}_\xi,\overline{\Delta}_c(\xi),\underline{\Delta}_c(\xi)$.

\begin{prop}\label{p-Lacoin2}
Let $\mathcal{G}$ be the family in Lemma \ref{l-Lacoin1} satisfying \textnormal{(\ref{eq-Lacoin2})} and $\overline{\Delta}_c(\xi)$, $\underline{\Delta}_c(\xi)$ be the constants in \textnormal{(\ref{eq-Deltac})}. Then, the following are equivalent.
\begin{itemize}
\item[(1)] For any increasing sequence $\xi$, $\overline{\Delta}_0(\xi)=\overline{\Delta}_\infty(\xi)$ and $\underline{\Delta}_0(\xi)=\underline{\Delta}_\infty(\xi)$.

\item[(2)] For any increasing sequence $\xi$, $\overline{\Delta}_0(\xi)=\overline{\Delta}_\infty(\xi)$ or $\underline{\Delta}_0(\xi)=\underline{\Delta}_\infty(\xi)$.

\item[(3)] $\mathcal{G}_c$ presents a total variation cutoff.
\end{itemize}

In particular, if $\underline{\Delta}_0=\overline{\Delta}_\infty=\Delta$, then $\mathcal{G}_c$ has a total variation cutoff with cutoff time $2(1+\Delta)^{-1}q_nn/\hat{p}_n$.
\end{prop}

\begin{rem}\label{r-Lacoin}
The monotonicity of $\overline{F},\underline{F}$ and the relation of $\underline{F}\le\overline{F}$ are clear from their definitions. These observations result in $\overline{\Delta}_0\le\overline{\Delta}_\infty$, $\underline{\Delta}_0\le\underline{\Delta}_\infty$, $\overline{\Delta}_0\le\underline{\Delta}_0$, and $\overline{\Delta}_\infty\le\underline{\Delta}_\infty$
\end{rem}

Concerning families without subfamilies presenting cutoffs, one may derive a proof similar to that of Proposition \ref{p-Lacoin2} to achieve the following corollary.

\begin{cor}\label{c-Lacoin}
Referring to the setting in Proposition \ref{p-Lacoin2}, the following are equivalent.
\begin{itemize}
\item[(1)] For any increasing sequence $\xi$, $\overline{\Delta}_0(\xi)<\overline{\Delta}_\infty(\xi)$ and $\underline{\Delta}_0(\xi)<\underline{\Delta}_\infty(\xi)$.

\item[(2)] For any increasing sequence $\xi$, $\overline{\Delta}_0(\xi)<\overline{\Delta}_\infty(\xi)$ or $\underline{\Delta}_0(\xi)<\underline{\Delta}_\infty(\xi)$.

\item[(3)] No subfamily of $\mathcal{G}_c$ has a total variation cutoff.
\end{itemize}
\end{cor}

We are now ready to state our example. Let $(\mathcal{X}_{n,i},K_{n,i},\pi_{n,i})$ and $(\mathcal{X}_n,K_n,\pi_n)$ be Markov chains in Lemma \ref{l-Lacoin1} satisfying
\[
 \max_{1\le i\le n}a_{n,i}=O\left(\frac{1}{n^2}\right),\quad \frac{1}{Cn}\le b_{n,i}\le \frac{C}{n},\quad \hat{p}_n\sim \check{p}_n,
\]
where $C>1$, $\check{p}_n=\max\{p_{n,i}|1\le i\le n\}$ and $\hat{p}_n=\min\{p_{n,i}|1\le i\le n\}$. Clearly, (\ref{eq-Lacoin2}) is fulfilled and the functions in (\ref{eq-Fxi}) satisfy
\[
 C^{-1}\le \underline{F}(\delta)\le\overline{F}(\delta)\le C,\quad\forall 0<\delta<1.
\]
By Corollary \ref{c-Lacoin}, no subfamily of $\mathcal{G}_c$ presents a total variation cutoff. Let $T_{n,\text{\tiny TV}}^{(c)}$ be the total variation mixing time of the $n$th chain in $\mathcal{G}_c$. It is easy to see from Lemmas \ref{l-Lacoin1} and \ref{l-comp2} that $T_{n,\text{\tiny TV}}^{(c)}(1/4)\asymp q_nn/\hat{p}_n$. Let $\mathcal{R}=(r_n)_{n=1}^\infty$, where $r_n=(q_nn/\hat{p}_n)\exp\{-n^\alpha\}$, and write
\[
 \log\frac{T_{n,\text{\tiny TV}}^{(c)}(1/4)}{r_n}=n^\alpha+O(1).
\]
Since $(n+1)^\alpha-n^\alpha\ge n^{\alpha-1}$, the above logarithm is increasing for $n$ large enough. As a result of Theorem \ref{t-main}, no subfamily of $\mathcal{G}^\mathcal{R}_c$ has a total variation cutoff.

Let $\xi_n=n(n+1)/2$ and $\mathcal{F}=(\mathcal{Y}_n,L_n,\nu_n)_{n=1}^\infty$, where
\[
 (\mathcal{Y}_{\xi_n+i},L_{\xi_n+i},\nu_{\xi_n+i})
 =(\mathcal{X}_{n+1,i},K_{n+1,i},\pi_{n+1,i}),\quad\forall 1\le i\le n+1,\,n\ge 0.
\]
By Proposition \ref{p-Lacoin2}, it is easy to see that $\mathcal{F}_c$ has a total variation cutoff with cutoff time $n$. Set $s_n=r_1+\cdots+r_n$, $u_{\xi_n+i}=p_{n+1,i}r_{n+1}/q_{n+1}$, $H_{n,i,t}=e^{-t(I-K_{n,i})}$, $H_{n,t}=e^{-t(I-K_n)}$ and $\widetilde{H}_{n,t}=e^{-t(I-L_n)}$. For simplicity, we write $\bigotimes_{i=1}^nA_i$ for the tensor product of matrices $A_1,...,A_n$. It is clear that $s_n=u_1+\cdots+u_{\xi_n}$. This implies
\begin{align}
 \bigotimes_{i=1}^{\xi_n}\widetilde{H}_{i,u_it/s_n}&=\bigotimes_{m=1}^n\left(
 \bigotimes_{i=1}^mH_{m,i,u_{\xi_{m-1}+i}t/s_n}\right)\notag\\
 &=\bigotimes_{m=1}^n\left(\bigotimes_{i=1}^mH_{m,i,p_{m,i}r_mt/(q_ms_n)}\right) =\bigotimes_{m=1}^nH_{m,r_mt/s_n}.\notag
\end{align}
By setting $\mathcal{U}=(u_n)_{n=1}^\infty$, the above identity implies that the subfamily of $(\mathcal{F}^\mathcal{U})_c$ indexed by $\xi$ is exactly $(\mathcal{G}^\mathcal{R})_c$ and, hence, has no cutoff in the total variation, as desired.

\section*{Acknowledgements}
We thank Laurent Saloff-Coste for recalling us that the Hellinger distance was used in the proof of Kakutani's dichotomy theorem. The first author is partially supported by grant MOST 104-2115-M-009-013-MY3 and by NCTS. The second author is partially supported by the Grant-in-Aid for Scientific Research (A) 25247007.

\appendix
\section{Auxiliary results}

\begin{lem}\label{l-analytic2}
Let $(s_n)_{n=1}^\infty$ and $\{\lambda_{n,i}|1\le i\le k_n,n\ge 1\}$ be a sequence and a triangular array of positive reals and set
\[
 F(c)=\lim_{m\ra\infty}\limsup_{n\ra\infty}\sum_{i=1}^{k_n-m}e^{-c\lambda_{n,i}s_n},
 \quad\forall c>0.
\]
Suppose $F(c_0)<\infty$ for some $c_0>0$. Then, either $F(c)>0$ for all $c>c_0$ or $F(c)=0$ for all $c>c_0$.
\end{lem}

To prove Lemma \ref{l-analytic2}, we need the following fact.

\begin{lem}{\rm(\cite[Lemma 3.2]{CSal10})}\label{l-analytic}
For $n\ge 1$, let $f_n$ be a function defined by
\[
 f_n(t)=\sum_{i=1}^\infty a_{n,i}e^{-t\lambda_{n,i}},\quad\forall t\ge 0,
\]
where $a_{n,i}\ge 0$ and $\lambda_{n,i+1}\ge\lambda_{n,i}>0$ for $i\ge 1$ and $n\ge 1$. Suppose $\sup_n f_n(0)<\infty$. Then, for any sequence of positive reals $(t_n)_{n=1}^\infty$, there is a subsequence $(t_{k_n})_{n=1}^\infty$ such that the sequence $g_n(c):=f_{k_n}(ct_{k_n})$ converges uniformly on any compact subset of $(0,\infty)$ to an analytic function on $(0,\infty)$.
\end{lem}

\begin{proof}[Proof of Lemma \ref{l-analytic2}]
Suppose $F(c_1)>0$ for some $c_1>c_0$. For $n>m$, set
\[
 f_{n,m}(c)=\sum_{i=1}^{k_n-m}e^{-c\lambda_{n,i}s_n}\quad\forall c>0.
\]
By the definition of $F(c_1)$, one may choose sequences $(n_j)_{j=1}^\infty,(m_j)_{j=1}^\infty$ satisfying $n_{j-1}<m_j<n_j$ such that
\[
 F(c_1)\le f_{n_j,m_j}(c_1)\le F(c_1)+2^{-j}\quad\forall j\ge 1.
\]
Define $g_j=f_{n_j,m_j}$. In this setting, it is clear that
\begin{equation}\label{eq-c1}
 \lim_{j\ra\infty}g_j(c_1)=F(c_1),\quad g_j\le f_{n_j,m}, \quad\forall m\le m_j.
\end{equation}
Note that the second inequality of (\ref{eq-c1}) implies
\begin{equation}\label{eq-gj}
 \limsup_{j\ra\infty}g_j(c)\le \lim_{m\ra\infty}\limsup_{j\ra\infty}f_{n_j,m}(c)
 \le\lim_{m\ra\infty}\limsup_{n\ra\infty}f_{n,m}(c)=F(c),
\end{equation}
which yields $\limsup_jg_j(c_0)\le F(c_0)<\infty$. In additional to the fact that $g_j(c_0)=f_{n_j,m_j}(c_0)\le k_{n_j}<\infty$ for all $j$, this leads to $\sup_jg_j(c_0)<\infty$. Next, by writing
\[
 g_j(c)=\sum_{i=1}^{n_j-m_j}e^{-c_0\lambda_{n_j,i}s_{n_j}}
 e^{-(c-c_0)\lambda_{n_j,i}s_{n_j}},\quad\forall c\ge c_0,
\]
we may select, by Lemma \ref{l-analytic}, a subsequence $(g_{\ell_j})_{j=1}^\infty$ such that
\[
 g_{\ell_j}\ra g\quad\text{uniformly on any compact subset of $(c_0,\infty)$},
\]
where $g$ is analytic on $(0,\infty)$. Consequently, (\ref{eq-c1}) implies $g(c_1)=F(c_1)>0$ and then (\ref{eq-gj}) leads to $F(c)\ge g(c)>0$ for all $c>c_0$ due to the analyticity of $g$.
\end{proof}

\begin{lem}{\rm(\cite[Corollary 3.3]{CSal10})}\label{l-analytic3}
Let $f_n$ be the function in Lemma \ref{l-analytic} and $t_n>0$. Assume that $f_n(0)$ is bounded and set, for $a>0$,
\[
 G(a)=\limsup_{n\ra\infty}f_n(at_n),\quad H(a)=\liminf_{n\ra\infty}f_n(at_n).
\]
Then, either $G(a)>0$ (resp. $H(a)>0$) for all $a>0$ or $G(a)=0$ (resp. $H(a)=0$) for all $a>0$.
\end{lem}

\section{Technical proofs}

\subsection{Proof of Proposition \ref{p-Lacoin3}}
Note that the assumption in (\ref{eq-Lacoin4}) fits the requirement in (\ref{eq-Lacoin2}). Referring to the notations in (\ref{eq-Fxi})-(\ref{eq-Deltac}), one has $\overline{F}=\underline{F}$, $\overline{\Delta}_0=\underline{\Delta}_0$ and $\overline{\Delta}_\infty=\underline{\Delta}_\infty$. It is easy to check that, for Case (1),
\[
 \overline{F}(\delta)=\begin{cases}\infty&\text{for }\beta\in(0,1),\\1&\text{for }\beta=1,\\0&\text{for }\beta\in(1,\infty),\end{cases}\quad
 \begin{cases}\overline{\Delta}_0=\overline{\Delta}_\infty=0&\text{for }\beta\in(0,1),\\\overline{\Delta}_0=0,\,\overline{\Delta}_\infty=1&\text{for }\beta=1,\\\overline{\Delta}_0=\overline{\Delta}_\infty=1&\text{for }\beta\in(1,\infty),\end{cases}
\]
and, for Case (2),
\[
 \overline{F}(\delta)=\begin{cases}\infty&\text{for }\beta\in(0,1),\\1&\text{for }\beta=\delta^{1/\alpha},\\0&\text{for }\beta\in(1,\infty),\end{cases}\quad
 \begin{cases}\overline{\Delta}_0=\overline{\Delta}_\infty=0&\text{for }\beta\in(0,1),\\\overline{\Delta}_0=0,\,\overline{\Delta}_\infty=1&\text{for }\beta=1,\\\overline{\Delta}_0=\overline{\Delta}_\infty=1&\text{for }\beta\in(1,\infty),\end{cases}
\]
and, for Case (3),
\[
 \overline{F}(\delta)=\begin{cases}0&\text{for }\beta\in(0,1),\,\delta\in(0,\beta)\\\infty&\text{for }\beta\in(0,1),\,\delta\in(\beta,1),\\0&\text{for }\beta\in[1,\infty),\end{cases}\quad
 \begin{cases}\overline{\Delta}_0=\overline{\Delta}_\infty=\beta&\text{for }\beta\in(0,1),\\\overline{\Delta}_0=\overline{\Delta}_\infty=1&\text{for }\beta\in[1,\infty).\end{cases}
\]
The desired result is then given by Proposition \ref{p-Lacoin2}, Corollary \ref{c-Lacoin} and the following additional observations,
\[
 q_n\sim\begin{cases}n&\text{for Case (1)},\\
 (\alpha+2)/(\alpha+1)&\text{for Case (2)},\\
 2n&\text{for Case (3)}.\end{cases}
\]

\subsection{Proofs of Lemma \ref{l-Lacoin1} and Proposition \ref{p-Lacoin2}}

First of all, we need the following lemma.

\begin{lem}\label{l-Lacoin}
Let $(\mathcal{X},K,\pi)$ be the chain in \textnormal{(\ref{eq-LacoinK})} with $(a_{n,i},b_{n,i}n^{-\beta},c_{n,i})=(a,b,c)$. Assume that $a<b$ and $a+b<1/2$. Then, one has, for $t>(n+1)/(1-2a)$,
\begin{equation}\label{eq-hdub1}
 d_H^{(c)}(t)^2\le 2at+b+e^{-t}\left(\frac{te}{n+1}\right)^{n+1}\frac{\sqrt{n+1}}{(1-2a)t-(n+1)},
\end{equation}
and, for $t>2n/(1-2a)$,
\begin{equation}\label{eq-hdub2}
 d_H^{(c)}(t)^2\le 2at+e^{-t}\left(\frac{te}{2n}\right)^{2n}\frac{\sqrt{2n}}{(1-2a)t-2n},
\end{equation}
and, for $n<t<2n$,
\begin{equation}\label{eq-hdlb1}
\begin{aligned}
 d_H^{(c)}(t)^2\ge&\frac{1}{2}\left[a+(1-a)^{2n}b\left(1-e^{-t}\left(\frac{te}{2n}\right)^{2n}
 \frac{\sqrt{2n}}{2n-t}\right)\right]\\
 &\quad-\sqrt{ab}(1-a)^n\left(1-e^{-t}\left(\frac{te}{2n}\right)^{2n}
 \frac{\sqrt{2n}}{2n-t}\right)^{1/2},
\end{aligned}
\end{equation}
and, for $0<t<n$,
\begin{equation}\label{eq-hdlb2}
 d_{\textnormal{\tiny TV}}^{(c)}(t)\ge 1-2a-e^{-t}\left(\frac{te}{n}\right)^n
 \frac{\sqrt{n}}{n-t}.
\end{equation}
\end{lem}

\begin{proof}[Proof of Lemma \ref{l-Lacoin}]
We first make some analysis on the stationary distribution and the discrete time chain. Note that $K$ is reversible and
\[
 \frac{\pi(i)}{\pi(0)}=\begin{cases}(1-a)^i/a^i&\text{for }0\le i\le n,\\
 (1-a)^{i-1}b/a^i&\text{for }n<i\le 2n.\end{cases}
\]
By the reversibility, one has
\[
 c=\frac{a^n(1-a-b)}{b(1-a)^{n-1}}<\frac{a^n}{b(1-a)^{n-2}}
\]
and
\[
 \pi(2n)=\frac{b(1-2a)(1-a)^{2n-1}}{b(1-a)^{2n}+(1-a-b)a^n(1-a)^n-a^{2n+1}}=
 \frac{1-2a}{1-a+d},
\]
where
\[
 d=c-\frac{ac}{1-a-b}\left(\frac{a}{1-a}\right)^n.
\]
Since $a<b$ and $a+b<1/2$ are assumed, it is easy to see that $0<d<c<a$, which leads to
\begin{equation}\label{eq-pi2n}
 1-2a<\pi(2n)<1-a.
\end{equation}
For the discrete time chain, let $(X_n)_{n=0}^\infty$ be a realization of $(\mathcal{X},K,\pi)$ and set
\[
 \begin{cases}
 A_m=\{X_m=2n\},\\
 B_m=\{X_j=i+j,\,\forall 0\le j\le n-i,\,X_j=2n,\,\forall n-i<j\le m\},\\
 C_m=\{X_j=i+j,\,\forall 0\le j\le 2n-i,\,X_j=2n,\,\forall 2n-i<j\le m\}.\end{cases}
\]
Given $\{X_0=i\}$, one has
\[
 A_m\supset\begin{cases}B_m&\text{for }0\le i\le n,\,n-i<m<2n-i,\\
 B_m\cup C_m&\text{for }\,0\le i\le n,\,m\ge 2n-i,\\
 C_m&\text{for }n<i\le 2n,\, m\ge 2n-i,\end{cases}
\]
and
\[
 \begin{cases}
 \mathbb{P}(B_m|X_0=i)=(1-a)^{n-i}(1-a-b)(1-a-c)^{m-n+i-1}&\text{for }0\le i\le n,\\
 \mathbb{P}(C_m|X_0=i)=(1-a)^{2n-i-1}b(1-a-c)^{m-2n+i}&\text{for }0\le i\le n,\\
 \mathbb{P}(C_m|X_0=i)=(1-a)^{2n-i}(1-a-c)^{m-2n+i}&\text{for }n<i\le 2n.\end{cases}
\]
Since $c<a<b<1/2$, we may conclude from the above computations that, for $0\le i\le 2n$,
\begin{equation}\label{eq-Kmlb}
 K^m(i,2n)\ge\begin{cases}(1-b)(1-2a)^m&\forall n<m<2n,\\
 (1-2a)^m&\forall m\ge 2n,\end{cases}
\end{equation}
where $1-a-b>(1-2a)(1-b)$ is used. Similarly, given $\{X_0=0\}$, if $0\le i\le n$, then $A_m=\emptyset$; if $n<i\le 2n$, then $A_m\subset\{X_i=i,\,\forall 0\le i\le m\}^c$. Both cases lead to
\begin{equation}\label{eq-Kmub}
 K^m(0,2n)\begin{cases}=0&\forall 0\le m\le n,\\
 \le 1-(1-a)^{m-1}b&\forall n<m\le 2n.\end{cases}
\end{equation}

Next, we consider the continuous time case and let $N_t$ be a Poisson process with parameter $1$. By (\ref{eq-Kmlb}) and (\ref{eq-Kmub}), it is easy to see that
\begin{equation}\label{eq-Htlb}
\begin{aligned}
 H_t(i,2n)&\ge(1-b)e^{-t}\sum_{m=n+1}^\infty\frac{[(1-2a)t]^m}{m!}
 +be^{-t}\sum_{m=2n}^\infty\frac{[(1-2a)t]^m}{m!}\\
 &=e^{-2at}\left[(1-b)\mathbb{P}\left(N_{(1-2a)t}>n\right)+b\mathbb{P}\left(N_{(1-2a)t}\ge 2n\right)\right],
\end{aligned}
\end{equation}
and
\begin{equation}\label{eq-Htub}
\begin{aligned}
 H_t(0,2n)&\le [1-(1-a)^{2n}b]\mathbb{P}(n<N_t\le 2n)+\mathbb{P}(N_t>2n)\\
 &=\mathbb{P}(N_t>n)-(1-a)^{2n}b\mathbb{P}(n<N_t\le 2n).
\end{aligned}
\end{equation}
Note that, for $t>n$,
\[
 \mathbb{P}(N_t<n)=e^{-t}\sum_{m=0}^{n-1}\frac{t^m}{m!}
 \le e^{-t}\frac{t^{n-1}}{(n-1)!}\sum_{m=0}^{n-1}\left(\frac{n}{t}\right)^m
 \le e^{-t}\frac{t^n}{(n-1)!(t-n)},
\]
and, for $t<n$,
\[
 \mathbb{P}(N_t\ge n)=e^{-t}\sum_{m=n}^\infty\frac{t^m}{m!}\le e^{-t}\frac{t^n}{n!}\sum_{m=n}^\infty\left(\frac{t}{n}\right)^m
 \le e^{-t}\frac{t^n}{(n-1)!(n-t)}.
\]
As one has $n!\ge n^{n+1/2}e^{-n}$, (\ref{eq-Htlb}) yields that, for $(1-2a)t>n+1$,
\begin{equation}\label{eq-Htlb1}
\begin{aligned}
 H_t(i,2n)&\ge (1-b)\left[e^{-2at}-e^{-t}\left(\frac{te}{n+1}\right)^{n+1}
 \frac{\sqrt{n+1}}{(1-2a)t-(n+1)}\right]\\
 &\ge 1-2at-b-e^{-t}\left(\frac{te}{n+1}\right)^{n+1}
 \frac{\sqrt{n+1}}{(1-2a)t-(n+1)},
\end{aligned}
\end{equation}
and, for $(1-2a)t>2n$,
\begin{equation}\label{eq-Htlb2}
\begin{aligned}
 H_t(i,2n)&\ge e^{-2at}-e^{-t}\left(\frac{te}{2n}\right)^{2n}
 \frac{\sqrt{2n}}{(1-2a)t-2n}\\
 &\ge 1-2at-e^{-t}\left(\frac{te}{2n}\right)^{2n}\frac{\sqrt{2n}}{(1-2a)t-2n}.
\end{aligned}
\end{equation}
In a similar way, one may use (\ref{eq-Htub}) to derive that, for $n<t<2n$,
\begin{equation}\label{eq-Htub1}
\begin{aligned}
 H_t(0,2n)&\le 1-(1-a)^{2n}b\mathbb{P}(N_t<2n)\\
 &\le 1-(1-a)^{2n}b\left[1-e^{-t}\left(\frac{te}{2n}\right)^{2n}
 \frac{\sqrt{2n}}{2n-t}\right],
\end{aligned}
\end{equation}
and, for $0<t<n$,
\begin{equation}\label{eq-Htub2}
 H_t(0,2n)\le\mathbb{P}(N_t\ge n)\le e^{-t}\left(\frac{te}{n}\right)^n
 \frac{\sqrt{n}}{n-t}.
\end{equation}

To finish the proof, we need some further inequalities. Let $\mu,\nu$ be probabilities on $\mathcal{X}$, $x_0\in\mathcal{X}$ and $A\subset\mathcal{X}$. By Lemma \ref{l-comp}, it is easy to see that
\[
 \|\mu-\nu\|_H^2\le\|\mu-\nu\|_{\text{\tiny TV}}\le 1-\mu(x_0)\wedge\nu(x_0)=
 (1-\mu(x_0))\vee(1-\nu(x_0)).
\]
By the Cauchy-Schwarz inequality, one has
\[
 \|\mu-\nu\|_H^2\ge\frac{1}{2}\sum_{x\in A}\left(\sqrt{\mu(x)}-\sqrt{\nu(x)}\right)^2\ge
 \frac{\mu(A)+\nu(A)}{2}-\sqrt{\mu(A)\nu(A)}.
\]
From the definition of total variation, it is obvious that $\|\mu-\nu\|_{\text{\tiny TV}}\ge \nu(x_0)-\mu(x_0)$. As a consequence of (\ref{eq-pi2n}) and (\ref{eq-Htlb1})-(\ref{eq-Htub2}), the desired inequalities are given by replacing $\mu,\nu,x_0,A$ with $H_t(i,\cdot),\pi,2n,\{0,1,...,2n-1\}$.
\end{proof}

\begin{proof}[Proof of Lemma \ref{l-Lacoin1}]
Let $d_{n,i,H}^{(c)}$ and $d_{n,H}^{(c)}$ be the Hellinger distances of the continuous time chains associated with $(\mathcal{X}_{n,i},K_{n,i},\pi_{n,i})$ and $(\mathcal{X}_n,K_n,\pi_n)$. For convenience, we set $a_n=\max\{a_{n,i}|1\le i\le n\}$ and $b_n=\max\{b_{n,i}|1\le i\le n\}$.

We first discuss (1). Let $C>1$ and $C'=(C+1)/2$. As it is assumed that $a_{n,1}+\cdots+a_{n,n}=o(1/n)$, one may select $N>0$ such that $Cn>C'(n+1)/(1-2a_{n,i})$ for all $1\le i\le n$ and $n\ge N$. By (\ref{eq-hdub2}) of Lemma \ref{l-Lacoin}, this implies
\[
 d_{n,i,H}^{(c)}(2Cn)^2\le 4a_{n,i}Cn+\frac{e^{2(\log C+1-C)n}}{(C'-1)\sqrt{2n}},\quad\forall 1\le i\le n,\,n\ge N,
\]
which yields that, for $n\ge N$,
\[
 \sum_{i=1}^nd_{n,i,H}^{(c)}(2Cn)^2\le 4Cn\sum_{i=1}^na_{n,i}+\frac{\sqrt{n}e^{2(\log C+1-C)n}}{\sqrt{2}(C'-1)}\ra 0,\quad\text{as }n\ra\infty.
\]
As a result of Proposition \ref{p-prodmixing}, one has the second limit in (\ref{eq-Lacoinprecutoff}). By Lemma \ref{l-comp}, to prove the first limit in (\ref{eq-Lacoinprecutoff}), it suffices to show the desired convergence in the total variation. Assume without loss of generality that $\hat{p}_n=p_{n,1}$ and let $d_{n,1,\text{\tiny TV}}^{(c)}$ be the total variation of the continuous chain associated with $(\mathcal{X}_{n,1},K_{n,1},\pi_{n,1})$. By Proposition \ref{p-prodmixing} and (\ref{eq-hdlb2}), we obtain
\[
 d_{n,\text{\tiny TV}}^{(c)}(C^{-1}q_nn/\hat{p}_n)\ge d_{n,1,\text{\tiny TV}}^{(c)}(C^{-1}n)\ge 1-2a_{n,1}-e^{[\log(1/C)+1-1/C]n}\ra 1,
\]
as $n\ra\infty$.

Next, we consider (2). Let $E_{n,\delta},B_{n,\delta}$ be as in Lemma \ref{l-Lacoin1} and $0<\Delta_-<\Delta<\Delta_+<1$. By Lemma \ref{l-Lacoin}, when $t>(n+1)/(1-2a_n)$ and $s>2n/(1-2a_n)$, one has
\begin{align}
 d_{n,i,H}^{(c)}(t)^2&\le \mathbf{1}_{\{t<s\}}\left(2a_{n,i}t+b_{n,i}+e^{-t}\left(\frac{te}{n+1}\right)^{n+1}
 \frac{\sqrt{n+1}}{(1-2a_{n,i})t-(n+1)}\right)\notag\\
 &\qquad+\mathbf{1}_{\{t\ge s\}}\left(2a_{n,i}s+e^{-s}\left(\frac{se}{2n}\right)^{2n}\frac{\sqrt{2n}}{(1-2a_{n,i})s-2n}\right)\notag\\
 &\le 2a_{n,i}s+\mathbf{1}_{\{t<s\}}b_{n,i}+g_n(t,s),\notag
\end{align}
where
\[
 g_n(t,s)=e^{-t}\left(\frac{te}{n+1}\right)^{n+1}
 \frac{\sqrt{n+1}}{(1-2a_n)t-(n+1)}+e^{-s}\left(\frac{se}{2n}\right)^{2n}\frac{\sqrt{2n}}{(1-2a_n)s-2n}.
\]
As $g$ is decreasing in $t$ for $t>(n+1)/(1-2a_n)$, the replacement of $t=2np_{n,i}/[(1+\Delta)\hat{p}_n]$ and $s=2n(1+\Delta_+)/(1+\Delta)$ in the above computations yields that, for $n$ large enough,
\[
 \max_{1\le i\le n}d_{n,i,H}^{(c)}\left(\frac{2np_{n,i}}{(1+\Delta)\hat{p}_n}\right)^2
 \le\frac{4(1+\Delta_+)a_nn}{1+\Delta}+b_n+g_n\left(\frac{2n}{(1+\Delta)},
 \frac{2n(1+\Delta_+)}{(1+\Delta)}\right),
\]
and
\begin{align}
 \sum_{i=1}^nd_{n,i,H}^{(c)}\left(\frac{2np_{n,i}}{(1+\Delta)\hat{p}_n}\right)^2
 \le&\frac{4(1+\Delta_+)n}{1+\Delta}\sum_{i=1}^na_{n,i}+B_n(\Delta_+)\notag\\
 &\qquad+ng_n\left(\frac{2n}{(1+\Delta)},\frac{2n(1+\Delta_+)}{(1+\Delta)}\right).\notag
\end{align}
It's an easy exercise to compute
\[
 \lim_{n\ra\infty}n^\alpha g_n\left(\frac{2n}{(1+\Delta)},\frac{2n(1+\Delta_+)}
 {(1+\Delta)}\right)=0,\quad\forall \alpha>0.
\]
As a consequence of (\ref{eq-Lacoin2}), this leads to
\[
 \lim_{n\ra\infty}\max_{1\le i\le n}d_{n,i,H}^{(c)}\left(\frac{2np_{n,i}}{(1+\Delta)\hat{p}_n}\right)^2=0,\quad\forall 0<\Delta<1,
\]
and
\[
 \sum_{i=1}^nd_{n,i,H}^{(c)}\left(\frac{2np_{n,i}}{(1+\Delta)\hat{p}_{n}}\right)^2
 \le B_n(\Delta_+).
\]
The upper bound in (\ref{eq-Lacoin3}) is then given by Proposition \ref{p-prodmixing}.

We prove the lower bound in a similar reasoning. Since $e^{-t}(\frac{te}{2n})^{2n}$ is increasing in $t$ for $0<t<2n$, one may use Lemma \ref{l-Lacoin} to derive
\[
 2d_{n,i,H}^{(c)}(t)^2\ge \mathbf{1}_{\{t<s\}}b_{n,i}h_n(s),
\]
for $t>n$ and $s<2n$, where
\[
 h_n(s)=(1-a_n)^{2n}\left(1-e^{-s}\left(\frac{se}{2n}\right)^{2n}\frac{\sqrt{2n}}{2n-s}\right)
 -2\sqrt{\max_{1\le i\le n}\frac{a_{n,i}}{b_{n,i}}}.
\]
Immediately, the replacement of $t=2np_{n,i}/[(1+\Delta)\hat{p}_n]$ and $s=2n(1+\Delta_-)/(1+\Delta)$ yields that, for $n$ large enough,
\[
 \sum_{i=1}^nd_{n,i,H}^{(c)}\left(\frac{2np_{n,i}}{(1+\Delta)\hat{p}_n}\right)^2\ge \frac{1}{2}B_n(\Delta_-)h_n\left(\frac{2n(1+\Delta_-)}{(1+\Delta)}\right).
\]
The desired lower bound in (\ref{eq-Lacoin3}) is then given by (\ref{eq-Lacoin2}) and Proposition \ref{p-prodmixing}.
\end{proof}

\begin{proof}[Proof of Proposition \ref{p-Lacoin2}]
(1)$\Ra$(2) is obvious. For (2)$\Ra$(3), we recall  \cite[Proposition 2.1]{CSal10}, which says that a family has a cutoff if and only if any subfamily has a further subfamily that presents a cutoff. Let $\xi=(\xi_n)_{n=1}^\infty$ be an increasing sequence of positive integers. Here, we discuss the case that $\overline{\Delta}_0(\xi)=\overline{\Delta}_\infty(\xi)=\overline{\Delta}$, while the other case can be shown in a similar way. Consider the following two subcases, (i) $\overline{\Delta}<1$ and (ii) $\overline{\Delta}=1$. In case (i), let $\overline{\delta}_n$ be a decreasing sequence in $(0,1)$ with limit $\overline{\Delta}$. Set $k_0=0$. For $n\ge 1$, since $\overline{F}(\overline{\delta}_n)=\infty$, one may select $k_n>k_{n-1}$ such that
$B_{\xi_{k_n}}(\overline{\delta}_n)>n$. Clearly, $B_{\xi_n}(\cdot)$ is non-decreasing on $(0,1)$. As a result, when $\overline{\Delta}<\delta<1$, we have
\[
 B_{\xi_{k_n}}(\delta)\ge B_{\xi_{k_n}}(\overline{\delta}_n)>n,\quad \text{for $n$ large enough.}
\]
By setting $\xi'_n=\xi_{k_n}$, the above inequalities imply $\overline{F}_{\xi'}(\delta)=\underline{F}_{\xi'}(\delta)=\infty$ for $\overline{\Delta}<\delta<1$. When $\overline{\Delta}>0$, it is obvious that $\underline{F}_{\xi'}(\delta)=\overline{F}_{\xi'}(\delta)\le\overline{F}_\xi(\delta)=0$ for $0<\delta<\overline{\Delta}$. Now, we show that $(\mathcal{G}_{\xi'})_c$ has a cutoff in the Hellinger distance. For $\delta\in(0,1)$, set $\delta'=(\delta+\overline{\Delta})/2$. By Lemma \ref{l-Lacoin1}(2), we obtain
\[
 \liminf_{n\ra\infty}d_{\xi_n',H}^{(c)}\left(\frac{2q_{\xi_n'}\xi_n'}
 {(1+\delta)\hat{p}_{\xi_n}}
 \right)^2\ge 1-e^{-\underline{F}_{\xi'}(\delta')/2}=1,\quad\forall \overline{\Delta}<\delta<1,
\]
and, for $\overline{\Delta}>0$,
\[
 \limsup_{n\ra\infty}d_{\xi_n',H}^{(c)}\left(\frac{2q_{\xi_n'}\xi_n'}
 {(1+\delta)\hat{p}_{\xi_n'}}
 \right)^2\le 1-e^{-\overline{F}_{\xi'}(\delta')}=0,\quad\forall 0<\delta<\overline{\Delta}.
\]
When $\overline{\Delta}=0$, Lemma \ref{l-Lacoin1}(1) yields
\[
 \lim_{n\ra\infty}d_{\xi_n',H}^{(c)}\left(\frac{2q_{\xi_n'}\xi_n'}{(1+\delta)\hat{p}_{\xi_n}}
 \right)^2=0,\quad\forall -1<\delta<\overline{\Delta}.
\]
Consequently, this proves that $(\mathcal{G}_{\xi'})_c$ has a cutoff in the Hellinger distance with cutoff time $2q_{\xi_n'}\xi_n'/[(1+\overline{\Delta})\hat{p}_{\xi_n'}]$. For case (ii), let $\delta'$ be the constant as before. By Lemma \ref{l-Lacoin1}, (\ref{eq-Lacoin3}) implies
\[
 \limsup_{n\ra\infty}d_{\xi_n,H}^{(c)}\left(\frac{2q_{\xi_n}\xi_n}{(1+\delta)\hat{p}_{\xi_n}}
 \right)^2\le 1-e^{-\overline{F}_{\xi}(\delta)}=0,\quad\forall 0<\delta<\overline{\Delta},
\]
while the former limit in (\ref{eq-Lacoinprecutoff}) yields
\[
 \lim_{n\ra\infty}d_{\xi_n,H}^{(c)}\left(\frac{2q_{\xi_n}\xi_n}
 {(1+\delta)\hat{p}_{\xi_n}}\right)^2=1,\quad\forall \delta>\overline{\Delta}.
\]
As a consequence, we prove that $(\mathcal{G}_\xi)_c$ has a cutoff in the Hellinger distance with cutoff time $2q_{\xi_n}\xi_n/[(1+\overline{\Delta})\hat{p}_{\xi_n}]$.
The total variation cutoffs of $(\mathcal{G}_{\xi'})_c$ and $(\mathcal{G}_\xi)_c$ are given by Proposition \ref{p-comp0}.

For (3)$\Ra$(1), it suffices to show that if $\overline{\Delta}_0(\xi)<\overline{\Delta}_\infty(\xi)$ or $\underline{\Delta}_0(\xi)<\underline{\Delta}_\infty(\xi)$ holds for some increasing sequence $\xi$, then $(\mathcal{G}_\xi)_c$ has a subfamily that presents no cutoff in the Hellinger distance, which is equivalent to no cutoff in the total variation. In the following, we deal with the case $\overline{\Delta}_0(\xi)<\overline{\Delta}_\infty(\xi)$, while the other case can be proved using a similar reasoning. By the definition of $\overline{F}_\xi$, one may select a subsequence of $\xi$, say $\xi''=(\xi''_n)_{n=1}^\infty$, and $0<A<B<1$ such that
\[
 \alpha:=\inf_{n\ge 1}B_{\xi_n''}(A)>0,\quad \beta:=\sup_{n\ge 1}B_{\xi_n''}(B)<\infty.
\]
Let $A',B'$ be constants satisfying $A<A'<B'<B$. By Lemma \ref{l-Lacoin1}(2), there is $N>0$ such that, for $n\ge N$,
\[
 d_{\xi_n'',H}^{(c)}\left(\frac{2q_{\xi_n''}\xi_n''}{(1+B')\hat{p}_{\xi_n''}}\right)\le 1-e^{-2B_{\xi_n}(B)}\le 1-e^{-2\beta}<1,
\]
and
\[
 d_{\xi_n'',H}^{(c)}\left(\frac{2q_{\xi_n''}\xi_n''}{(1+A')\hat{p}_{\xi_n''}}\right)\ge 1-e^{-B_{\xi_n}(A)/4}\ge 1-e^{-\alpha/4}>0.
\]
This implies that no subfamily of $(\mathcal{G}_{\xi''})_c$ presents a cutoff in the Hellinger distance and finishes the proof of the equivalences.

The sufficiency for cutoffs in the specific case follows immediately from (2), while the proof for the cutoff time is similar to the proof of (1) and skipped.
\end{proof}

\end{document}